\documentclass[english]{amsart}
\usepackage[T1]{fontenc}
\usepackage[latin9]{inputenc}
\usepackage{float}
\usepackage{units}
\usepackage{amsbsy}
\usepackage{amstext}
\usepackage{amsthm}
\usepackage{amssymb}
\usepackage{graphicx}
\usepackage{setspace}

\makeatletter

\newcommand{\noun}[1]{\textsc{#1}}
\providecommand{\tabularnewline}{\\}

\numberwithin{equation}{section}
\numberwithin{figure}{section}
\theoremstyle{plain}
\newtheorem{thm}{\protect\theoremname}
\theoremstyle{plain}
\newtheorem{prop}[thm]{\protect\propositionname}
\theoremstyle{definition}
\newtheorem{defn}[thm]{\protect\definitionname}
\theoremstyle{plain}
\newtheorem{lem}[thm]{\protect\lemmaname}
\theoremstyle{remark}
\newtheorem{rem}[thm]{\protect\remarkname}
\theoremstyle{definition}
\newtheorem{example}[thm]{\protect\examplename}
\theoremstyle{plain}
\newtheorem{cor}[thm]{\protect\corollaryname}
\theoremstyle{definition}
\newtheorem*{defn*}{\protect\definitionname}
\theoremstyle{plain}
\newtheorem*{prop*}{\protect\propositionname}
\theoremstyle{definition}
\newtheorem*{example*}{\protect\examplename}
\theoremstyle{remark}
\newtheorem*{rem*}{\protect\remarkname}

\usepackage[colorlinks,urlcolor=blue,citecolor=blue,filecolor=blue,linkcolor=blue]{hyperref}
\makeatletter
\def\l@subsection{\@tocline{2}{0pt}{2.5pc}{5pc}{}}
\makeatother

\makeatother

\usepackage{babel}
\providecommand{\corollaryname}{Corollary}
\providecommand{\definitionname}{Definition}
\providecommand{\examplename}{Example}
\providecommand{\lemmaname}{Lemma}
\providecommand{\propositionname}{Proposition}
\providecommand{\remarkname}{Remark}
\providecommand{\theoremname}{Theorem}

\begin{document}
\title{Multiple facets of inverse continuity}
\author{Szymon Dolecki}
\address{Institut de Mathématiques de Bourgogne, Université de Bourgogne et
Franche Comté, Dijon, France}
\thanks{I am grateful to Dr \noun{Alois Lechicki} of Fürth for a conscientious
perusal of the manuscript, resulting in sagacious improvements of
the presentation.}
\date{\today}
\begin{abstract}
Inversion of various inclusions that characterize continuity in topological
spaces results in numerous variants of quotient and perfect maps.
In the framework of convergences, the said inclusions are no longer
equivalent, and each of them characterizes continuity in a different
concretely reflective subcategory of convergences. On the other hand,
it turns out that the mentioned variants of quotient and perfect maps
are quotient and perfect maps with respect to these subcategories.
This perspective enables use of convergence-theoretic tools in quests
related to quotient and perfect maps, considerably simplifying the
traditional approach. Similar techniques would be unconceivable in
the framework of topologies.
\end{abstract}

\maketitle
\global\long\def\lm{\lim\nolimits}%
\global\long\def\it{\operatorname{int}\nolimits}%
\global\long\def\ih{\operatorname{inh}\nolimits}%
\global\long\def\ad{\operatorname{adh}\nolimits}%
\global\long\def\cl{\operatorname{cl}\nolimits}%
\global\long\def\ih{\operatorname{inh}\nolimits}%
\global\long\def\card{\mathrm{card\,}}%
\global\long\def\0{\mathrm{\varnothing}}%

\section*{Introduction}

This paper is designed for a broad mathematical audience. Its purpose
is to show the utility of the convergence theory approach to classical
themes of general topology. Therefore, I focus on convergence-theoretic
methods rather than on detailed applications. For this reason, I shall
provide only a small number of examples, referring for more details
to classical books on general topology, for example, to \cite{Eng}
of \noun{R. Engelking}, and to research papers cited below.

It will be shown that the arguments deployed would be impossible without
a broader framework transcending that of topologies.

Continuity of maps between topological spaces can be characterized
in many ways, in particular, in terms of adherences of filters from
various classes, their images, and preimages.

It appears that inversion of the inclusions occurring in these characterizations
results in numerous variants of quotient and perfect maps. Extending
these definitions to general convergences leads, on one hand, to quotients
with respect to various reflective subcategories, like topologies,
pretopologies, paratopologies, and pseudotopologies, and, on the other,
to fiber relations having compactness type related to the listed convergence
categories. This vision enables to see the quintuple quotient quest
of \noun{E. Michael} \cite{quest} as a single unifying quotient quest
presented by the author in \cite{quest2}.

It turns out that continuity is not pertinent to many considerations
involving quotient-like and perfect-like maps. In fact, the latter
properties correspond to certain stability features of fiber relations,
that is, of the inverse relations of mappings, rather than to mappings
themselves, and can be studied advantageously independently from continuity. 

Accordingly, no continuity assumption will be implicitly made. Let
me also say that no separation axiom will be assumed in the forthcoming
definitions, unless explicitly stated.

Convergences (in particular, pseudotopologies) were introduced and
investigated by \noun{G. Choquet} in \cite{cho}, because topologies
turned out to be inadequate in the study of hyperspaces (\footnote{More precisely, of upper and lower limits, said of \emph{Kuratowski};
in fact, they were introduced by \noun{G. Peano} in \cite{peano87},\cite{peano_1908}.}). The category of convergences is \emph{exponential} (\footnote{The category of convergences with continuous maps as morphisms is
\emph{exponential} (or \emph{Cartesian closed}), which means that
for two given convergences $\xi,\sigma$, there exists a \emph{coarsest
convergence} (denoted by $\left[\xi,\sigma\right]$) on the set of
continuous functions $C(\xi,\sigma)$, for which the canonical coupling
$\left\langle x,f\right\rangle :=f(x)$ is continuous: $\left\langle ,\right\rangle \in C(\xi\times\left[\xi,\sigma\right],\sigma)$.
In particular, hyperconvergences are spaces of continuous maps with
respect to $\sigma=\$$, the \emph{Sierpi\'{n}ski topology} on a two
elements set, say $\{0,1\}$, the open sets of which are $\varnothing,\{1\},\{0,1\}$.
The category of topologies with continuous maps as morphisms is not
exponential, as for two topologies $\xi,\sigma$, the convergence
$\left[\xi,\sigma\right]$ need not be topological.}), contrary to the category of topologies. This feature of convergences
made them suitable for the just mentioned study of hyperspaces. But
their usefulness proved to be much broader, as I am going to show
in this paper.

Throughout this paper, $f^{-}$ stands for the \emph{inverse relation}
of a \emph{map} $f$, hence we use the convention that
\[
f^{-}(B):=f^{-1}(B)
\]
denotes the \emph{preimage} of $B$ under $f$. 

To deal with many cross-references more easily, Table of Contents
has been placed at the end of the paper, not at its beginning, because
it is very detailed, and thus lengthy.

\section{Continuity in topological spaces}

Consider non-empty sets $X$ and $Y$ (\footnote{We shall assume that the sets considered are non-empty unless stated
otherwise explicitly.}), topologies $\xi$ on $X$ and $\tau$ on $Y$, a map $f:X\rightarrow Y$,
and $A\subset X$. By definition, $f$ is \emph{continuous} (from
$\xi$ to $\tau$) if the preimage under $f$ of each $\tau$-open
set is $\xi$-open. We denote by $C(\xi,\tau)$ the set of all functions
that are continuous from $\xi$ to $\tau$, by $\cl_{\xi}A$ the \emph{closure},
and by $\it_{\xi}A$ the \emph{interior} of $A$ with respect to $\xi$,
and so on.

The following equivalences can be found in \cite[Theorem 1.4.1]{Eng},
but constitute also a standard exercise in topology courses.
\begin{prop}
\label{prop:1}Let $\xi$ and $\tau$ be topologies. Then the following
statements are equivalent:
\begin{gather}
f\in C(\xi,\tau),\label{eq:1}\\
\cl_{\xi}f^{-}(B)\subset f^{-}(\cl_{\tau}B)\textrm{ for each }B\subset Y,\label{eq:2}\\
f(\cl_{\xi}A)\subset\cl_{\tau}f(A)\textrm{ for each }A\subset X.\label{eq:3}
\end{gather}
\end{prop}

If we invert the two inclusions above (\footnote{These conditions are also equivalent to $f^{-}(\mathrm{int}_{\tau}B)\subset\mathrm{int}_{\xi}f^{-}(B)$
for each $B\subset Y$ (see \cite[Theorem 1.4.1]{Eng}); the inversion
of this inclusion $\mathrm{int}_{\xi}f^{-}(B)\subset f^{-}(\mathrm{int}_{\tau}B)$
turns out to be equivalent to (\ref{eq:5}). However, a formally similar
condition $\it_{\tau}f(A)\subset f(\it_{\xi}A)$ is unrelated to continuity.}), then we get
\begin{gather}
f^{-}(\cl_{\tau}B)\subset\cl_{\xi}f^{-}(B)\textrm{ for each }B\subset Y,\label{eq:4}\\
\cl_{\tau}f(A)\subset f(\cl_{\xi}A)\textrm{ for each }A\subset X.\label{eq:5}
\end{gather}
The properties so obtained can be described as inverse to (\ref{eq:2})
and (\ref{eq:3}), respectively. It turns out that they are not equivalent!
The first characterizes \emph{quotient maps}, and the second,\emph{
closed maps}. Recall that a map $f\in C(\xi,\tau)$ is said to be
\begin{enumerate}
\item \emph{quotient }if $B$ is $\tau$-closed whenever $f^{-}(B)$ is
$\xi$-closed, 
\item \emph{closed} if $f(A)$ is $\tau$-closed for every $\xi$-closed
set $A$.
\end{enumerate}
In the sequel, we will consider only surjective maps, a non-essential
restriction that simplifies the presentation. The definitions and
conditions above make sense without the continuity assumption and,
moreover, the continuity assumption is superfluous in the following
characterizations.

We denote by $\left|\xi\right|$ the underlying set of a topology
$\xi$. Hence, here $\left|\xi\right|=X$ and $\left|\tau\right|=Y$.
\begin{prop}
\label{prop:quot-perf} If a map $f:\left|\xi\right|\longrightarrow\left|\tau\right|$
is surjective, then (\ref{eq:4}) holds if and only if $f$ is quotient.
\end{prop}

\begin{proof}
If $f^{-}(B)$ is $\xi$-closed, that is, $\cl_{\xi}f^{-}(B)\subset f^{-}(B)$,
which together with (\ref{eq:4}) yields $f^{-}(\cl_{\tau}B)\subset f^{-}(B)$,
hence $\cl_{\tau}B\subset B$, because $f$ is surjective. Conversely,
if $f$ is quotient and $\cl_{\xi}f^{-}(B)=f^{-}(B)$, then $\cl_{\tau}B=B$,
hence $f^{-}(\cl_{\tau}B)=f^{-}(B)$, and thus (\ref{eq:4}). 
\end{proof}
\begin{prop}
If a map $f:\left|\xi\right|\longrightarrow\left|\tau\right|$ is
surjective, then (\ref{eq:5}) holds if and only if $f$ is closed.
\end{prop}

\begin{proof}
If $\cl_{\xi}A\subset A$, by (\ref{eq:5}), $\cl_{\tau}f(A)\subset f(A)$,
that is, $f(A)$ is $\tau$-closed. Conversely, if $\cl_{\xi}A=A$
implies that $\cl_{\tau}f(A)\subset f(A)$, hence (\ref{eq:5}).
\end{proof}
Let us now enlarge our perspective by making use of the concept of
filter, which will enable us to formulate new characterizations of
continuity.

\subsection{Convergence of filters in topological spaces}

One of the ways of defining a topology is to specify its converging
\emph{filters}. A non-empty family $\mathcal{F}$ of subsets of a
set $X$ is called a \emph{filter} whenever
\begin{gather*}
(F_{0}\in\mathcal{F})\wedge(F_{1}\in\mathcal{F})\Longleftrightarrow F_{0}\cap F_{1}\in\mathcal{F}.
\end{gather*}
A filter $\mathcal{F}$ is called \emph{non-degenerate} if 
\[
\varnothing\notin\mathcal{F}.
\]
Basic facts about filters are gathered in Appendix \ref{subsec:Filters}.
Let us only recall here that the family $\mathcal{N}_{\theta}(x)$
(of\emph{ neighborhoods} of $x$ for a topology $\theta$) is a non-degenerate
filter.

Convergence of filters in topological spaces generalizes convergence
of sequences. Indeed, the family $\mathcal{O}_{\theta}(x)$ of all
$\theta$-open sets that contain $x$, is a \emph{base} of the neighborhood
filter $\mathcal{N}_{\theta}(x)$ (see Appendix \ref{subsec:Filters}).
On the other hand, a sequence $(x_{n})_{n}$ converges to $x$ in
$\theta$ whenever, for each $\theta$-open set $O$ containing $x$,
there is $m\in\mathbb{N}$ such that $\{x_{n}:n\geq m\}\subset O$,
that is, whenever $O$ belongs to the filter \emph{generated} (Appendix
\ref{subsec:Sequential-filters}) by
\[
\{\{x_{n}:n\geq m\}:m\in\mathbb{N}\}.
\]

\begin{defn}
\label{def:conv}If $\theta$ is a topology on a set $X$, then we
say that a \emph{filter} $\mathcal{F}$ \emph{converges} to a point
$x$,
\[
x\in\lm_{\theta}\mathcal{F},
\]
whenever $\mathcal{O}_{\theta}(x)\subset\mathcal{F}$. As $\mathcal{O}_{\theta}(x)$
is a base of the \emph{neighborhood filter} $\mathcal{N}_{\theta}(x)$
of $x$ for $\theta$, we get
\begin{equation}
x\in\lm_{\theta}\mathcal{F}\Longleftrightarrow\mathcal{N}_{\theta}(x)\subset\mathcal{F}.\label{eq:fliter_conv}
\end{equation}
\end{defn}

More generally, if $\mathcal{B}$ is a base of $\mathcal{F}$, then
by definition $\lm_{\theta}\mathcal{B}=\lm_{\theta}\mathcal{F}$.

\subsection{Some operations on families of sets}

The symbol $\mathcal{A}^{\uparrow}$ for a given family $\mathcal{A}$
(\footnote{If $\mathcal{A}=\varnothing$ then $\varnothing^{\uparrow}=\varnothing$.})
of subsets of a set $X$ stands for
\begin{equation}
\mathcal{A}^{\uparrow}:=\bigcup_{A\in\mathcal{A}}\nolimits\{H\subset X:H\supset A\}.\tag{isotonization}\label{eq:isotonization}
\end{equation}
If $\mathcal{A}$ and $\mathcal{D}$ are families, then $\mathcal{A}$
is said to be \emph{coarser} than $\mathcal{D}$ ($\mathcal{D}$ \emph{finer}
than $\mathcal{A}$)
\[
\mathcal{A}\leq\mathcal{D}
\]
if for each $A\in\mathcal{A}$ there is $D\in\mathcal{D}$ such that
$D\subset A.$ The \emph{grill} $\mathcal{A}^{\#}$ of a family $\mathcal{A}$
on $X$ is defined by \noun{G. Choquet} in \cite{cho_grille} by
\[
\mathcal{A}^{\#}:=\{H\subset X:\underset{A\in\mathcal{A}}{\forall}\;H\cap A\neq\varnothing\}.
\]
If $\mathcal{A}$ and $\mathcal{D}$ are families, then we write $\mathcal{A}\#\mathcal{D}$
($\mathcal{H}$ \emph{meshes} $\mathcal{A}$) whenever $\mathcal{D}\subset\mathcal{A}^{\#}$,
equivalently, $\mathcal{A}\subset\mathcal{D}^{\#}$. This means that
$H\cap A\neq\varnothing$ for each $H\in\mathcal{H}$ and each $A\in\mathcal{A}$. 

Let $\mathcal{A}$ be a family of subsets of $X$, $\mathcal{G}$
a family of subsets of $Y,$ and $f:X\rightarrow Y.$ Then (\footnote{If $\mathcal{A}$ and $\mathcal{B}$ are filters, then $f\left[\mathcal{A}\right]$
and $f^{-}[\mathcal{B}]$ are filter-bases. })
\[
f\left[\mathcal{A}\right]:=\{f(A):A\in\mathcal{A}\},\qquad f^{-}\left[\mathcal{B}\right]:=\{f^{-}(B):B\in\mathcal{B}\}.
\]
It is straightforward that if $f:X\rightarrow Y$ is surjective (\footnote{Actually, the first inequality is valid for any map.}),
$\mathcal{G}$ and $\mathcal{H}$ are filters, then
\begin{equation}
\mathcal{G}\geq f^{-}[f[\mathcal{G}]],\qquad f[f^{-}[\mathcal{H}]]=\mathcal{H}.\label{eq:f-fsurH}
\end{equation}

In the sequel, we make continually use of the formula (see Appendix
\ref{subsec:Relations})
\begin{equation}
f\left[\mathcal{A}\right]\#\mathcal{B}\Longleftrightarrow\mathcal{A}\#f^{-}\left[\mathcal{B}\right].\label{eq:grill_dual}
\end{equation}

\subsection{Adherences\label{subsec:Adherences}}

Let $\xi$ be a topology on a set $X$. The \emph{adherence} of a
family $\mathcal{A}$ of subsets of $X$ is defined by (\footnote{It is implicit that $\mathcal{H}$ is either a filter or a filter-base
on $X$, because the operation $\lm_{\xi}\mathcal{H}$ is defined
only for such $\mathcal{H}$.})
\begin{equation}
\ad_{\xi}\mathcal{A}:=\bigcup\nolimits _{\mathcal{H}\#\mathcal{A}}\lm_{\xi}\mathcal{H}.\label{eq:adh}
\end{equation}

If $\mathcal{A}$ is a filter, then $\ad_{\xi}\mathcal{A}=\bigcup\nolimits _{\mathcal{H}\supset\mathcal{A}}\lm_{\xi}\mathcal{H}$
is an alternative formula for (\ref{eq:adh}). In particular, we abridge
$\ad_{\xi}A:=\ad_{\xi}\{A\}^{\uparrow}$ (see (\ref{eq:isotonization})
above). Accordingly
\begin{equation}
\mathcal{A}\leq\mathcal{D}\Longrightarrow\ad_{\xi}\mathcal{D}\subset\ad_{\xi}\mathcal{A}.\label{eq:anti}
\end{equation}

\begin{lem}
\label{lem:open} If $\xi$ is a topology on a set $X$, then $A$
is $\xi$-open if and only if $A\in\mathcal{F}$ for every filter
$\mathcal{F}$ such that $A\cap\lm_{\xi}\mathcal{F}\neq\varnothing$
\emph{(}\footnote{If $A$ is open, then the condition holds by definition. If $A$ is
not $\xi$-open, then there exists $x\in A$ such that $O\setminus A\neq\varnothing$
for each $O\in\mathcal{O}_{\xi}(x)$. In other words, $A^{c}:=(X\setminus A)\in\mathcal{O}_{\xi}(x)^{\#}$,
which means that $x\in\ad_{\xi}A^{c}$. Therefore, there exists a
filter $\mathcal{F}$ such that $A^{c}\in\mathcal{F}$ and $x\in\lm_{\xi}\mathcal{F}$.
Obviously, $A\notin\mathcal{F}$.}\emph{)}.
\end{lem}

\subsection{Back to continuity}

In terms of filters, $f\in C(\xi,\tau)$ whenever
\begin{equation}
f(\lm_{\xi}\mathcal{F})\subset\lm_{\tau}f\left[\mathcal{F}\right]\label{eq:conv_cont}
\end{equation}
for each filter $\mathcal{F}$ (\footnote{Let $f\in C(\xi,\tau)$, and $x\in\lm_{\xi}\mathcal{F}$. To prove
that $f(x)\in\lm_{\tau}f\left[\mathcal{F}\right]$, let $O\in\mathcal{O}_{\tau}(f(x))$.
By continuity, $f^{-}(O)\in\mathcal{O}_{\xi}(x)$ and thus $f^{-}(O)\in\mathcal{F}$.
Consequently, $O\supset f(f^{-}(O))\in f[\mathcal{F}]$, that is,
$f(x)\in\lm f\left[\mathcal{F}\right]$.

Conversely, assume (\ref{eq:conv_cont}) for each filter $\mathcal{F}$
and let $O$ be a $\tau$-open set. By Lemma \ref{lem:open}, to show
that $f^{-}(O)$ is $\xi$-open, let $x\in f^{-}(O)$ and take a filter
$\mathcal{F}$ such that $x\in\lm_{\xi}\mathcal{F}$. By assumption,
$f(x)\in\lm_{\tau}f\left[\mathcal{F}\right]$ and $f(x)\in O$, hence
$O\in f\left[\mathcal{F}\right]$. It follows that $x\in f^{-}(f(x))\subset f^{-}(O)\in f^{-}[f[\mathcal{F}]]\leq\mathcal{F}$,
hence $f^{-}(O)\in\mathcal{F}$.}).

We shall now characterize continuity in terms of adherences of filters.
Let us remind that a filter on $X$ is called principal if it is of
the form $\{A\}^{\uparrow}:=\{F\subset X:A\subset F\}$, where $A\subset X$.

We write $\mathcal{F}\in\mathbb{F}$ whenever $\mathcal{F}$ is a
filter, and $\mathcal{F}\in\mathbb{F}X$ if $\mathcal{F}$ is a filter
on $X$. In other terms, $\mathbb{F}$ is the class of all filters,
and $\mathbb{F}X$ the set of all filters on a set $X$. It follows
from \emph{Set Theory} that $\mathbb{F}$ is not a set. A class $\mathbb{J}$
is a \emph{subclass} of $\mathbb{F}$ whenever $\mathcal{J}\in\mathbb{J}$
implies that there is a set $X$ such that $\mathcal{J}\in\mathcal{\mathbb{F}}X$.
Then we write $\mathbb{J}\subset\mathbb{F}$, and we understand that
$\mathbb{J}X=\mathcal{\mathbb{J}\cap\mathbb{F}}X$.
\begin{prop}
\label{prop:2} Let $\xi$ and $\tau$ be topologies on $X$ and $Y$,
respectively. Let $\mathbb{H}$ and $\mathbb{G}$ be classes of filters
that contain all principal filters. Then the following statements
are equivalent: 
\begin{enumerate}
\item $f\in C(\xi,\tau)$,
\item for each $\mathcal{H}\in\mathbb{H}Y$,
\[
\ad_{\xi}f^{-}[\mathcal{H}]\subset f^{-}(\ad_{\tau}\mathcal{H}),
\]
\item for each $\mathcal{G}\in\mathbb{G}X$,
\[
f(\ad_{\xi}\mathcal{G})\subset\ad_{\tau}f[\mathcal{G}].
\]
\end{enumerate}
\end{prop}

\begin{proof}
It is enough to prove these equivalences for the class $\mathbb{F}$
of all filters. Indeed, Proposition \ref{prop:1} establishes the
analogous equivalence for principal filters. These two arguments entail
Proposition \ref{prop:2}, because $\mathbb{H}$ and $\mathbb{G}$
are classes of filters and both contain all principal filters.

Let (1) and $x\in\ad_{\xi}f^{-}[\mathcal{H}]$. Then there exists
a filter $\mathcal{F}$ such that $x\in\lm_{\xi}\mathcal{F}$ and
$\mathcal{F}\#f^{-}[\mathcal{H}]$, equivalently $f[\mathcal{F}]\#\mathcal{H}$.
By continuity, $f(x)\in\lm_{\tau}f\left[\mathcal{F}\right]$, hence
$f(x)\in\ad_{\tau}\mathcal{H}$ because of $f[\mathcal{F}]\#\mathcal{H}$,
that is, $x\in f^{-}(\ad_{\tau}\mathcal{H})$, which is (2).

Let $\mathcal{H}:=f[\mathcal{G}]$ and apply $f$ to the inclusion
in (2), getting
\[
f(\ad_{\xi}f^{-}[f[\mathcal{G}]])\subset f(f^{-}(\ad_{\tau}f[\mathcal{G}])),
\]
and by (\ref{eq:f-fsurH}) and (\ref{eq:anti}), $\ad_{\xi}\mathcal{G}\subset\ad_{\xi}f^{-}[f[\mathcal{G}]]$,
and thus
\[
f(\ad_{\xi}\mathcal{G})\subset f(\ad_{\xi}f^{-}[f[\mathcal{G}]])\subset f(f^{-}(\ad_{\tau}f[\mathcal{G}]))\subset\ad_{\tau}f[\mathcal{G}].
\]

Let (3) and let $A\subset X$. Set $\mathcal{G}:=\{A\}^{\uparrow}$.
Then we get $f(\cl_{\xi}A)\subset\cl_{\tau}f(A)$, that is, (3) of
Proposition \ref{prop:1}. It follows that $f\in C(\xi,\tau)$.
\end{proof}
Let us reconsider Proposition \ref{prop:2} in a broader context of
convergence spaces, which will give us a new perspective, unavailable
within the framework of topological spaces.

\section{Framework of convergences}

Topological spaces constitute a subclass of convergence spaces (see
the foundational paper \cite{cho} of \noun{G. Choquet}, and references
\cite{dolecki.BT},\cite{CFT}).

A relation $\xi$ between (non-degenerate) filters $\mathcal{F}$
on $X$ and points of $X$ (\protect\footnote{As $\xi\subset\mathbb{F}X\times X\text{,}$ a usual notation for the
image of an element $\mathcal{F}$ of $\mathbb{F}X$ under $\xi$
is $\xi\left(\mathcal{F}\right)\text{.}$ A purpose of our notation,
$\lm_{\xi}\mathcal{F}$ instead of $\xi\left(\mathcal{F}\right)\text{,}$
is to emphasize the particular role of this relation, which should
enhance readability.})
\begin{equation}
x\in\lm_{\xi}\mathcal{F}\label{eq:conv}
\end{equation}
 is called a \emph{convergence} provided that, for each $x,\mathcal{F}_{0}$
and $\mathcal{F}_{1}$, 
\begin{gather}
\mathcal{F}_{0}\subset\mathcal{F}_{1}\Longrightarrow\lm_{\xi}\mathcal{F}_{0}\subset\lm_{\xi}\mathcal{F}_{1},\label{isotone}\\
x\in\lm_{\xi}\left\{ x\right\} ^{\uparrow}.\label{centered}
\end{gather}

If (\ref{eq:conv}) holds, then we say that $x$ \emph{is a limit
(point) of} $\mathcal{F}$ or $\mathcal{F}$ \emph{converges to} $x\text{}$.
We denote by $\left\vert \xi\right\vert $ the \emph{underlying set}
of $\xi$, that is, $\left\vert \xi\right\vert =X$ (\protect\footnote{An underlying set can be restored from a convergence. In fact, $\left\vert \xi\right\vert =\bigcup\nolimits _{\mathcal{F}\in\mathbb{F}}\lm_{\xi}\mathcal{F},$
beacuse $\xi$ is a convergence on $X\text{,}$ then $X\supset\bigcup_{\mathcal{F}\in\mathbb{F}}\lm_{\xi}\mathcal{F}\supset\bigcup_{x\in X}\lm_{\xi}\left\{ x\right\} ^{\uparrow}\supset X\text{.}$ }). 

A convergence $\zeta$ is \emph{finer} than a convergence $\xi$ (equivalently,
$\xi$ is coarser than $\zeta$)
\[
\zeta\geq\xi,
\]
if $\lm$$_{\zeta}\mathcal{F}\subset\lm$$_{\xi}\mathcal{F}$ for
every filter $\mathcal{F}$. The ordered set of convergences on a
given set is a complete lattice, that is, each non-empty set of convergences
(on a given set) admits infima and suprema, 
\begin{gather}
\lm_{\bigvee\Xi}\mathcal{F}=\bigcap_{\xi\in\Xi}\lm_{\xi}\mathcal{F},\qquad\lm_{\bigwedge\Xi}\mathcal{F}=\bigcup_{\xi\in\Xi}\lm_{\xi}\mathcal{F}.\label{eq:sup}
\end{gather}

\subsection{Open and closed sets for arbitrary convergences}

Convergence of filters in a topological space was defined in Definition
\ref{def:conv} in terms of open sets, which were then characterized
in Lemma \ref{lem:open} in terms of filters. Now, open sets will
be defined for an arbitrary convergence!
\begin{defn}
Let $\xi$ be a convergence. A subset $O$ of $\left|\xi\right|$
is said to be \emph{$\xi$-open} (with respect to $\xi$) if
\[
O\cap\lm_{\xi}\mathcal{F}\neq\varnothing\Longrightarrow O\in\mathcal{F}.
\]
A subset of $\left|\xi\right|$ is said to be \emph{$\xi$-closed}
if its complement is $\xi$-open. For each subset $A$ of $\left|\xi\right|$,
there exists the least $\xi$-closed set $\cl_{\xi}A$ including $A$,
called the $\xi$-closure of $A$.
\end{defn}

The family of all $\xi$-open sets is denoted by $\mathcal{O}_{\xi}$,
and that of $\xi$-open sets containing $x$ by $\mathcal{O}_{\xi}(x)$.
Notice that $\mathcal{O}_{\xi}(x)$ is a filter base. The filter that
it generates is called the \emph{neighborhood filter} and is denoted
by $\mathcal{N}_{\xi}(x)$. 

We notice that the family $\mathcal{O}_{\xi}$ of all $\xi$-open
sets fulfills the axioms of open sets of a topology (\footnote{That is, $\0$ and $X$ are open, each union and each finite intersection
of open sets is open.}). Denote this topology by $\mathrm{T}\xi$. It is straightforward
that $\zeta\geq\xi$ implies $\mathrm{T}\zeta\geq\mathrm{T}\xi$;
moreover, $\mathrm{T}\xi\leq\xi$, and $\mathrm{T}\xi\leq\mathrm{T}(\mathrm{T}\xi)$
for every convergence $\xi$.

A convergence $\theta$ is called a \emph{topology} if
\[
\mathcal{O}_{\theta}(x)\subset\mathcal{F}\Longrightarrow x\in\lm_{\theta}\mathcal{F}
\]
for each $x$ and $\mathcal{F}$, the converse implication being true
for an arbitrary convergence.

It is obvious that a convergence $\theta$ is a topology if and only
if $\theta\leq\mathrm{T}\theta$.

\subsection{Continuity in convergence spaces}

Let $\xi$ and $\tau$ be convergences on $X$ and $Y$, respectively.
Let $f:X\rightarrow Y$ be a map. We say that $f$ is \emph{continuous}
from $\xi$ to $\tau$ if (\ref{eq:conv_cont})
\[
f(\lm_{\xi}\mathcal{F})\subset\lm_{\tau}f\left[\mathcal{F}\right]
\]
holds for each filter $\mathcal{F}$ on $X$ (\footnote{Continuity determines the notions of \emph{subconvergence}, \emph{quotient}
convergence, \emph{product} convergence, and so on. More precisely,
a convergence $\zeta$ is called a subconvergence $\xi$ if the \emph{identity
map} $i:\left|\zeta\right|\rightarrow\left|\xi\right|$ is an \emph{embedding},
that is, a \emph{homeomorphism} from $\left|\zeta\right|$ to $i(\left|\zeta\right|)$.
The quotient convergence has already been defined. If $\Xi$ is a
set of convergences, then the product convergence is the initial convergence
on $\prod_{\xi\in\Xi}\left|\xi\right|$ with respect to the projections
$p_{\theta}:\prod_{\xi\in\Xi}\left|\xi\right|\rightarrow\left|\theta\right|$,
that is, the coarsest convergence, for which $p_{\theta}$ is continuous
to $\theta$ for each $\theta\in\Xi$.}).

For each convergence $\xi$ on $X\text{,}$ there exists the finest
among the convergences $\tau$ on $Y$ for which $f\in C(\xi,\tau)$.
It is denoted by $f\xi$ and called the \emph{final} convergence.
Therefore, $f\in C(\xi,\tau)$ whenever $f\xi\geq\text{\ensuremath{\tau} }$.

On the other hand, for each convergence $\tau$ on $Y$, there exists
the coarsest among the convergences $\xi$ on $X$, for which $f\in C(\xi,\tau)$.
It is denoted by $f^{-}\tau$ and called the \emph{initial} convergence
for $(f,\tau)\text{.}$ Therefore, $f\in C(\xi,\tau)$ whenever $\xi\geq f^{-}\text{\ensuremath{\tau} }$.

Summarizing
\begin{equation}
f\xi\geq\tau\Longleftrightarrow f\in C\left(\xi,\tau\right)\Longleftrightarrow\xi\geq f^{-}\tau.\label{eq:summarize}
\end{equation}

Let me point out here that $f^{-}\tau$ is a topology provided that
$\tau$ is a topology, while the topologicity of $\xi$ does not entail
that of $f\xi$ (!).

Let us write down formulas for adherences in initial and final convergences
for surjective maps (\footnote{The formula for the adherence of the initial map also holds if $f$
is not surjective.}), which will be used in the sequel (see \cite[p. 215]{CFT}):
\begin{gather}
\ad_{f\xi}\mathcal{H}=f(\ad_{\xi}f^{-}[\mathcal{H}]),\qquad\ad_{f^{-}\tau}\mathcal{G}=f^{-}(\ad_{\tau}f[\mathcal{G}]).\label{eq:in_fin}
\end{gather}

\subsection{Functors, reflectors, coreflectors\label{subsec:Functors}}

From the category theory viewpoint, convergences are objects and continuous
maps are morphisms of a category, which is concrete over the category
of sets. Therefore, it is possible, and handy, to apply category theory
objectwise. Functors are certain maps defined on classes of morphisms,
and then specialized to the classes of objects viewed as identity
morphisms. Here it is enough to define functors directly on objects.

We say that that an operator $H$ associating to each convergence
$\xi$, a convergence $H\xi$ on $\left|\xi\right|$, is called a
\emph{concrete functor }(\footnote{The term \emph{concrete} means that $\left|H\xi\right|=\left|\xi\right|$.})
if, for any convergences $\zeta,\xi$ 
\begin{gather}
\zeta\geq\xi\Longrightarrow H\zeta\geq H\xi,\tag{isotone}\label{eq:iso}
\end{gather}
and 
\begin{equation}
C(\xi,\tau)\subset C(H\xi,H\tau),\tag{functorial}\label{eq:functorial}
\end{equation}
that is, each map continuous from $\xi$ to $\tau$, remains continuous
from $H\xi$ to $H\tau$. It is straightforward that if $H$ fulfills
(\ref{eq:iso}), then (\ref{eq:functorial}) holds if and only if,
for each convergence $\tau$ and each map $f$,
\[
H(f^{-}\tau)\geq f^{-}(H\tau),
\]
equivalently, for each convergence $\xi$ and each map $f$,
\[
f(H\xi)\geq f(H\xi).
\]

A concrete functor is called a \emph{reflector }whenever
\begin{gather}
H(H\xi)=H\xi,\tag{idempotent}\label{eq:idemp}\\
H\xi\leq\xi.\tag{contractive}\label{eq:contr}
\end{gather}
The topologizer $\mathrm{T}$, the pretopologizer $\mathrm{S}_{0}$
and the pseudotopologizer $\mathrm{S}$ are reflectors. A concrete
functor is called a \emph{coreflector }if (\ref{eq:idemp}) and
\begin{equation}
\xi\leq H\xi.\tag{expansive}\label{eq:expa}
\end{equation}

Other examples of reflectors and coreflectors are given in Appendix
\ref{subsec:Reflective-subclasses} and \ref{subsec:Corefelective-subclasses}.

\subsection{Adherence-determined reflectors\label{subsec:Adherence-determined-reflectors}}

The adherence $\ad_{\xi}\mathcal{A}$ of a family $\mathcal{A}$ with
respect to a convergence $\xi$ is defined, as in the topological
case, by (\ref{eq:adh}). In general, $\ad_{\xi}\mathcal{A}\subset\bigcap_{A\in\mathcal{A}}\cl_{\xi}A$,
the equality holding if $\xi$ is a topology. Hence, for a subset
$A$ of a topological space, $\ad_{\xi}A=\cl_{\xi}A$, which means
that, in topological spaces, adherence coincides with the usual closure.

We say that a class of filters $\mathbb{H}$ is \emph{initial} if
$f^{-}[\mathcal{H}]\in\mathbb{H}$ for each $\mathcal{H}\in\mathbb{H}$.
An operator $H$ is called \emph{adherence-determined} if there exists
an initial class of filters $\mathbb{H}$ containing all principal
filters and such that

\begin{equation}
\lm_{H\theta}\mathcal{F}=\bigcap_{\mathbb{H}\ni\mathcal{H}\#\mathcal{F}}\ad_{\theta}\mathcal{H}.\label{eq:adh-determined}
\end{equation}
It turns out that each adherence-determined operator is a concrete
reflector. Moreover, for each filter $\mathcal{H},$

\begin{equation}
\mathcal{H}\in\mathbb{H}\Longrightarrow\ad_{H\theta}\mathcal{H}=\ad_{\theta}\mathcal{H}\label{eq:adh_det}
\end{equation}
 We shall consider here the cases of
\begin{enumerate}
\item \emph{pretopologies }($\mathbb{H}=\mathbb{F}_{0}$ is the class of
\emph{principal filters}, and $H=\mathrm{S}_{0}$ is the \emph{pretopologizer}),
\item \emph{paratopologies} ($\mathbb{H}=\mathbb{F}_{1}$ is the class of
\emph{countably based filters}, and $H=\mathrm{S}_{1}$ is the \emph{paratopologizer}),
\item \emph{pseudotopologies} ($\mathbb{H}=\mathbb{F}$ is the class of
all filters, and $H=\mathrm{S}$ is the \emph{pseudotopologizer}).
\end{enumerate}
Topologies are adherence-determined in a generalized sense (see Appendix
\ref{subsec:Reflective-subclasses}).

\section{Characterizations of continuity}

We are now in a position to revisit Proposition \ref{prop:2} in the
context of convergences. It turns out that now the characterizations
in terms of adherences of filters $\mathcal{G}$ and $\mathcal{H}$
depend on the classes over which they range.

A class $\mathbb{J}$ of filters is said to be \emph{transferable}
(\footnote{In a broader context (see \cite{CFT}), they are called $\mathrm{\mathbb{F}}_{0}$-composable.})
if $\mathcal{J}\in\mathbb{J}$ implies that $R[\mathcal{J}]\in\mathbb{J}$
(see Appendix \ref{subsec:Relations}) for each relation $R$. In
particular, maps and their inverse relations are relations. Therefore,
transferable classes are initial. The classes of all filters, countably
based filters, and principal filters are transferable.
\begin{thm}
\label{thm:cont_in_conv} Let $\mathbb{J}$ be a transferable class
of filters. If $J$ is a $\mathbb{J}$-adherence-determined reflector
(\ref{eq:adh-determined}), then the following statements are equivalent:
\begin{enumerate}
\item $f\in C(J\xi,J\tau)$,
\item for each $\mathcal{H}\in\mathbb{J}$, 
\begin{equation}
\ad_{\xi}f^{-}[\mathcal{H}]\subset f^{-}(\ad_{\tau}\mathcal{H}),\label{eq:f-ad}
\end{equation}
\item for each $\mathcal{G}\in\mathbb{J}$, 
\begin{equation}
f(\ad_{\xi}\mathcal{G})\subset\ad_{\tau}f[\mathcal{G}].\label{eq:adf}
\end{equation}
\end{enumerate}
\end{thm}

\begin{proof}
If $f\in C(J\xi,J\tau)$, that is, $f\xi\geq f(J\xi)\geq J\tau$,
hence, for each filter $\mathcal{H}$,
\begin{equation}
\ad_{f\xi}\mathcal{H}\subset\ad_{J\tau}\mathcal{H},\label{eq:H}
\end{equation}
and if $\mathcal{H}\in\mathbb{J}$, by (\ref{eq:adh_det}), $f(\ad_{\xi}f^{-}[\mathcal{H}])=\ad_{f\xi}\mathcal{H}\subset\ad_{\tau}\mathcal{H}$.
On applying $f^{-}$ to this inequality, we get (\ref{eq:f-ad}).

If $\mathcal{G}\in\mathbb{J}$, then $\mathcal{H}:=f[\mathcal{G}]\in\mathcal{\mathbb{J}}$
and $f^{-}[\mathcal{H}]=f^{-}[f[\mathcal{G}]]\leq\mathcal{G}$. Thus
(\ref{eq:f-ad}) implies $\ad_{\xi}\mathcal{G}\subset f^{-}(\ad_{\tau}f[\mathcal{G}])$,
thus on applying $f$, we recover (\ref{eq:adf}).

On applying $f^{-}$ to (\ref{eq:adf}), we assure that, for each
filter $\mathcal{G}\in\mathbb{J}$,
\begin{equation}
\ad_{\xi}\mathcal{G}\subset\ad_{f^{-}\tau}\mathcal{G}.\label{eq:G}
\end{equation}
Hence, for each filter $\mathcal{F}$,
\[
\lm_{\xi}\mathcal{F}\subset\bigcap_{\mathbb{J}\ni\mathcal{G}\#\mathcal{F}}\ad_{\xi}\mathcal{G}\subset\bigcap_{\mathbb{J}\ni\mathcal{G}\#\mathcal{F}}\ad_{f^{-}\tau}\mathcal{G}=\lm_{J(f^{-}\tau)}\mathcal{F},
\]
that is, $\xi\geq J(f^{-}\tau)$, hence $J\xi\geq J(f^{-}\tau)$ $\geq f^{-}(J\tau)$,
because $J$ is a reflector.
\end{proof}
\begin{rem}
In the case of topologies, continuity is equivalent to the formulas
above for every class $\mathbb{J}$ of filters that contains all principal
filters (see Propositions \ref{prop:1} and \ref{prop:2}, and the
comments at the end of the section). This is because $\mathrm{T}\leq J$
for every $\mathbb{J}$-adherence-determined reflector corresponding
to such a class $\mathbb{J}$.
\end{rem}

\section{Inverse continuity}
\begin{defn}
\label{def:classes}A surjective map $f:\left|\xi\right|\rightarrow\left|\tau\right|$
is called $\mathbb{H}$\emph{-quotient} if
\begin{equation}
f^{-}(\ad_{\tau}\mathcal{H})\subset\ad_{\xi}f^{-}[\mathcal{H}]\label{eq:quot_rough}
\end{equation}
holds for each $\mathcal{H}\in\mathbb{H}$, and it is called $\mathbb{G}$\emph{-perfect}
if
\begin{equation}
\ad_{\tau}f[\mathcal{G}]\subset f(\ad_{\xi}\mathcal{G})\label{eq:per}
\end{equation}
holds for each $\mathcal{G}\in\mathbb{G}$.
\end{defn}

Observe that (\ref{eq:quot_rough}) and (\ref{eq:per}) are obtained
from (\ref{eq:f-ad}) and (\ref{eq:adf}), respectively, by reversing
the inclusion.
\begin{prop}
\label{prop:transferable}Let $\mathbb{J}$ be a transferable class
of filters. If a surjective map is $\mathbb{J}$-perfect then it is
$\mathbb{J}$-quotient. 
\end{prop}

\begin{proof}
Assume (\ref{eq:per}), $\mathcal{H}\in\mathbb{J}$ and $x\in f^{-}(\ad_{\tau}\mathcal{H})$.
Let $\mathcal{G}:=f^{-}[\mathcal{H}]\in\mathbb{J}$ by transferability,
and $f[f^{-}[\mathcal{H}]]=\mathcal{H}$, because $f$ is surjective.
Then, by (\ref{eq:per}), $x\in f^{-}(\ad_{\tau}\mathcal{H})\subset f^{-}(f(\ad_{\xi}f^{-}[\mathcal{H}]))\subset\ad_{\xi}f^{-}[\mathcal{H}]$,
which is (\ref{eq:quot}).
\end{proof}
The converse is not true, as shows Example \ref{exa:classic}.
\begin{rem}
\label{rem:Iper-quot}If $\mathbb{J}$ is transferable, then a bijective
map is $\mathbb{J}$-quotient if and only if it is $\mathbb{J}$-perfect.
\end{rem}

In order to better appreciate the sense of the defined notions, let
us express them in terms of adherences with respect to final and initial
convergences. By using (\ref{eq:in_fin}), we conclude that (\ref{eq:quot_rough})
is equivalent to
\begin{equation}
\ad_{\tau}\mathcal{H}\subset\ad_{f\xi}\mathcal{H},\label{eq:quot}
\end{equation}
which is the inversion of (\ref{eq:H}), whereas (\ref{eq:per}) is
equivalent to
\begin{equation}
\ad_{f^{-}\tau}\mathcal{G}\subset f^{-}(f(\ad_{\xi}\mathcal{G})),\label{eq:perf}
\end{equation}
which is not the inversion of (\ref{eq:G})!
\begin{rem}
Contemplating this asymmetry, remember that the initial convergence
can be seen as a convergence of fibers rather than that of individual
points. I mean by \emph{convergence of fibers of $f$} the property
that $x\in\lm_{f^{-}\tau}\mathcal{F}$ implies that
\[
f^{-}(f(x))\subset\lm_{f^{-}\tau}f^{-}[f[\mathcal{F}]].
\]
Indeed, $x\in\lm_{f^{-}\tau}\mathcal{F}$ amounts to $f(x)\in\lm_{\tau}f[\mathcal{F}]$,
and thus $f^{-}(f(x))\subset f^{-}(\lm_{\tau}f[\mathcal{F}])=\lm_{f^{-}\tau}\mathcal{F}$.
On the other hand, $f(f^{-}(f(A)))=f(A)$ for each (not necessarily
surjective) map $f$ and each $A$, and thus

\[
\lm_{f^{-}\tau}f^{-}[f[\mathcal{F}]]=f^{-}(\lm_{\tau}f[f^{-}[f[\mathcal{F}]]])=f^{-}(\lm_{\tau}f[\mathcal{F}])=\lm_{f^{-}\tau}\mathcal{F}.
\]
\end{rem}

\begin{thm}
\label{thm:quotient}\cite{quest2} If $H$ is a reflector adherence-determined
by $\mathbb{H}$, then a map $f:\left|\xi\right|\rightarrow\left|\tau\right|$
is $\mathbb{H}$\emph{-quotient} if and only if
\begin{equation}
\tau\geq H(f\xi).\label{eq:quotient}
\end{equation}
\end{thm}

\begin{proof}
By (\ref{eq:adh-determined}), (\ref{eq:quot}) holds for each $\mathcal{H}\in\mathbb{H}$,
if and only if, for each filter $\mathcal{F}$,
\[
\lm_{\tau}\mathcal{F}\subset\lm_{H\tau}\mathcal{F}=\bigcap_{\mathbb{H}\ni\mathcal{H}\#\mathcal{F}}\ad_{\tau}\mathcal{H}\subset\bigcap_{\mathbb{H}\ni\mathcal{H}\#\mathcal{F}}\ad_{f\xi}\mathcal{H}=\lm_{H(f\xi)}\mathcal{F}.
\]
\end{proof}

Formula (\ref{eq:quotient}) reveals that $\mathbb{H}$-quotient maps
are in fact quotient maps with respect to a reflector $H$ that is
\emph{adherence-determined} by $\mathbb{H}$. This way, (\ref{eq:quotient})
extends the definition to arbitrary concrete reflectors: $f$ is $H$\emph{-quotient}
whenever (\ref{eq:quotient}) holds. It is practical to use either
the term $H$\emph{-quotient} or $\mathbb{H}$\emph{-quotient} when
$H$ is adherence-determined by $\mathbb{H}$.

An important example of a non-adherence-determined reflector is the
identity reflector $\mathrm{I}$, the corresponding reflective class
of which is that of all convergences (\footnote{$\mathrm{I}$ is not adherence-determined, because the greatest adherence-determined
reflector is the pseudotopologizer $\mathrm{S}$, which adherence-determined
by the class $\mathbb{F}$ of all filters. Indeed, there exist convergences
that are not pseudotopologies (Example \ref{exa:non-pseudo}).}). Of course, a surjective map $f:\left|\xi\right|\rightarrow\left|\tau\right|$
is an $\mathrm{I}$-\emph{quotient map, }whenever
\begin{equation}
\tau\geq f\xi.\label{eq:conv_quot}
\end{equation}

Traditionally, $\mathrm{I}$-quotient maps are called \emph{almost
open} \cite{AVA1966}. It follows from the definition that (\footnote{Indeed, as $f\xi$ is the finest convergence on $\left|\tau\right|$,
for which $f$ is continuous, $y\in\lm_{f\xi}\mathcal{H}$ whenever
there exists $x\in f^{-}(y)$ and a filter $\mathcal{F}$ such that
$x\in\lm_{\xi}\mathcal{F}$ and $f[\mathcal{F}]=\mathcal{H}$.})
\begin{prop}
\label{prop:almost-open} A surjective map $f:\left|\xi\right|\rightarrow\left|\tau\right|$
is almost open if and only if for each filter $\mathcal{H}$ and each
$y\in\lm_{\tau}\mathcal{H}$, there exists $x\in f^{-}(y)$ and a
filter $\mathcal{F}$ such that $x\in\lm$$_{\xi}\mathcal{F}$ and
$f[\mathcal{F}]=\mathcal{H}$.
\end{prop}

Recall that a surjective map $f:\left|\xi\right|\rightarrow\left|\tau\right|$
is called \emph{open} provided that $f(O)$ is $\tau$-open for each
$\xi$-open subset $O$ of $\left|\xi\right|$. We shall see that
each open map (between topologies) is almost open.
\begin{prop}
\label{prop:open} A surjective map $f:\left|\xi\right|\rightarrow\left|\tau\right|$
is open provided that for each filter $\mathcal{H}$ and each $y\in\lm$$_{\tau}\mathcal{H}$
and for every $x\in f^{-}(y)$ there exists a filter $\mathcal{F}$
such that $x\in\lm$$_{\xi}\mathcal{F}$, and $f[\mathcal{F}]=\mathcal{H}$. 
\end{prop}

\begin{proof}
Suppose that the condition holds, and $O$ is $\xi$-open. If $y\in f(O)\cap\lm_{\tau}\mathcal{H}$,
then there exists $x\in f^{-}(y)\cap O$. On the other hand, there
exists a filter $\mathcal{F}$ such that $x\in\lm$$_{\xi}\mathcal{F}$,
and $f[\mathcal{F}]=\mathcal{H}$. As $O$ is $\xi$-open, $O\in\mathcal{F}$
and thus $f(O)\in f[\mathcal{F}]=\mathcal{H}$, proving that $f(O)$
is $\tau$-open.
\end{proof}
If $\xi$ is a topology, then the converse also holds \cite[p. 398]{CFT}.

\section{Characterizations in terms of covers\label{sec:Characterizations-in-terms}}

\subsection{Covers, inherences, adherences}

Open covers are not an adequate tool in general convergences, because
open sets do not determine non-topological convergences. Therefore,
the notion of cover need be extended.

Let $\xi$ be a convergence. A family $\mathcal{P}$ of sets is a
\emph{$\xi$-cover} of $A$ if $\mathcal{P}\cap\mathcal{F}\neq\varnothing$
for every filter $\mathcal{F}$ such that $A\cap\lm_{\xi}\mathcal{F}\neq\varnothing$. 

The concept of \emph{cover} is closely related to those of \emph{adherence}
and \emph{inherence}. If $\mathcal{P}$ be a family of subsets of
$X$, then $\mathcal{P}_{c}:=\{X\setminus P:P\in\mathcal{P}\}.$ The
inherence $\ih_{\xi}\mathcal{P}$ of a family $\mathcal{P}$ is defined
by
\begin{equation}
\ih_{\xi}\mathcal{P}:=(\ad_{\xi}\mathcal{P}_{c})^{c}.\label{eq:inh}
\end{equation}

Observe that if $\mathcal{P}$ a family and $\mathcal{F}$ is a filter,
then
\begin{equation}
\mathcal{P}\cap\mathcal{F}=\varnothing\Longleftrightarrow\mathcal{P}_{c}\#\mathcal{F}.\label{eq:grill}
\end{equation}

The following theorem constitute a special case of \cite[Theorem 3.1]{D.comp}.
\begin{thm}[{\cite[p. 69]{CFT}}]
\label{thm:duality} The following statements are equivalent:
\begin{enumerate}
\item $\mathcal{P}$ is a $\xi$-cover of $A$,
\item $A\subset\ih_{\xi}\mathcal{P}$,
\item $\ad_{\xi}\mathcal{P}_{c}\cap A=\varnothing$.
\end{enumerate}
\end{thm}

\begin{proof}
By definition, (1) holds whenever $\mathcal{P}\cap\mathcal{F}\neq\varnothing$
for every filter $\mathcal{F}$ such that $A\cap\lm_{\xi}\mathcal{F}\neq\varnothing$.
Equivalently, if $\mathcal{P}\cap\mathcal{F}=\varnothing$ then $A\cap\lm_{\xi}\mathcal{F}=\varnothing$
for each filter $\mathcal{F}$. By (\ref{eq:grill}), $\mathcal{P}_{c}\#\mathcal{F}$
implies that $A\cap\lm_{\xi}\mathcal{F}=\varnothing$, hence (1) and
(3) are equivalent. Now (3) amounts to $(\ih_{\xi}\mathcal{P})^{c}\cap A=\varnothing$,
that is, to (2). 
\end{proof}
Consequently, we shall use the symbol
\[
A\subset\ih_{\xi}\mathcal{P}
\]
for ``$\mathcal{P}$ is a $\xi$-cover of $A$''.
\begin{rem}
\label{rem:cover_top}In order to relate this notion to open covers,
observe that if $\xi$ is a topology, then $\mathcal{P}$ is a \emph{$\xi$-}cover
of $A$ whenever $A\subset\bigcup_{P\in\mathcal{P}}\it_{\xi}P,$ where
$\it_{\xi}P$ is the $\xi$-\emph{interior} of $P$. 
\end{rem}

\subsection{Characterizations}

Proposition \ref{prop:rephrase} below rephrases Definition \ref{def:classes}
in terms of the original definitions, introduced in \cite{quest2,D.comp}.
\begin{prop}
\label{prop:rephrase} A surjective map $f$ is $\mathbb{H}$-quotient
if and only if, for each $\mathcal{H}\in\mathbb{H}$ and each $y$,
\begin{equation}
y\in\ad_{\tau}\mathcal{H}\Longrightarrow f^{-}(y)\cap\ad_{\xi}f^{-}[\mathcal{H}]\neq\varnothing;\label{eq:quot_form}
\end{equation}
it is $\mathbb{G}$-perfect if and only if, for each $\mathcal{G}\in\mathbb{G}$
and each $y$,
\begin{equation}
y\in\ad_{\tau}f[\mathcal{G}]\Longrightarrow f^{-}(y)\cap\ad_{\xi}\mathcal{G}\neq\varnothing.\label{eq:perf_form}
\end{equation}
\end{prop}

We are now in a position to characterize quotient-like and perfect-like
maps in terms of covers.

We shall see that for a transferable class $\mathbb{H}$, a surjective
map $f$ is $\mathbb{H}$-quotient if and only if the image under
$f$ of every cover (from the dual class of $\mathbb{H}$) of $f^{-}(y)$
is a cover of $y$. More precisely,
\begin{prop}
\label{prop:quot_cov} Let $\mathbb{H}$ be transferable. A surjective
map $f$ is $\mathbb{H}$-quotient if and only if, for each $\mathcal{Q}$
such that $\mathcal{Q}_{c}\in\mathbb{H}$ and each $y$,
\begin{equation}
\ih_{\xi}\mathcal{Q}\supset f^{-}(y)\Longrightarrow y\in\ih_{\tau}f[\mathcal{Q}].\label{eq:qu_cov}
\end{equation}
\end{prop}

\begin{proof}
Indeed, set $\mathcal{Q}=f^{-}[\mathcal{H}_{c}]=f^{-}[\mathcal{H}]_{c}$
in (\ref{eq:quot_form}). Then $\mathcal{H}_{c}=f[f^{-}[\mathcal{H}_{c}]]=f[\mathcal{Q}]$.

By (\ref{eq:inh}), $y\in\ad_{\tau}f[\mathcal{Q}]_{c}$ implies $f^{-}(y)\cap\ad_{\xi}\mathcal{Q}_{c}\neq\varnothing$,
we infer that $y\notin\ih_{\tau}f[\mathcal{Q}]$ implies $f^{-}(y)\cap(\ih_{\xi}\mathcal{Q})^{c}\neq\varnothing$,
which is equivalent to (\ref{eq:perf_form}).
\end{proof}
We shall now see that if $\mathbb{G}$ is transferable, then $f$
is $\mathbb{G}$-perfect if and only if $\mathcal{P}$ is a cover
of $y$ provided that the preimage under $f$ of a family $\mathcal{P}$
(from the dual class of $\mathbb{G}$) is a cover of $f^{-}(y)$.
More precisely,
\begin{prop}
\label{prop:perf_cov} Let $\mathbb{G}$ be transferable. A surjective
map $f$ is $\mathbb{G}$-perfect if and only if, for each $\mathcal{P}$
such that $\mathcal{P}_{c}\in\mathbb{G}$ and each $y$,
\begin{equation}
\ih_{\xi}f^{-}[\mathcal{P}]\supset f^{-}(y)\Longrightarrow y\in\ih_{\tau}\mathcal{P}.\label{eq:per_cov}
\end{equation}
\end{prop}

\begin{proof}
In fact, by setting $\mathcal{G}=f^{-}[\mathcal{P}_{c}]=f^{-}[\mathcal{P}]_{c}$,
hence $f[\mathcal{G}]=f[f^{-}[\mathcal{P}_{c}]]=\mathcal{P}_{c}$.
Thus (\ref{eq:perf_form}) becomes $y\in\ad_{\tau}\mathcal{P}_{c}\Longrightarrow f^{-}(y)\cap\ad_{\xi}f^{-}[\mathcal{P}]_{c}\neq\varnothing$.
By (\ref{eq:inh}), $y\notin\ih_{\tau}\mathcal{P}$ implies $f^{-}(y)\cap(\ih_{\xi}f^{-}[\mathcal{P}])^{c}\neq\varnothing$,
that is, (\ref{eq:per_cov}).
\end{proof}
\begin{rem}
Notice that in (\ref{eq:per_cov}) and (\ref{eq:qu_cov}) above, covers
are either ideals (closed under finite unions and subsets) or ideal
bases.
\end{rem}

\section{Variants of quotient maps}

Many variants of quotient maps were defined in topology, the profusion
being due to a search for most adequate classes of maps that preserve
various important types of topological spaces \cite{quest}. I will
list the traditional names of $\mathbb{H}$-quotients (see Proposition
\ref{prop:rephrase}) for specific classes $\mathbb{H}$ of filters
(in the brackets):
\begin{enumerate}
\item \emph{quotient} (principal filters of $f\xi$-closed sets), 
\item \emph{hereditarily quotient} (principal filters),
\item \emph{countably biquotient }(countably based filters), 
\item \emph{biquotient} (all filters).
\end{enumerate}
In topology, these notions were originally defined in terms of covers.
Moreover, they comprised continuity. Remember that we do not ask that
the considered maps be continuous! 

We shall list below (in topological context) cover definitions of
quotient-like maps. By Theorem \ref{thm:duality} and Proposition
\ref{prop:quot_cov}, a surjective map $f:\left|\xi\right|\rightarrow\left|\tau\right|$
is
\begin{enumerate}
\item \emph{quotient} if a set is $\tau$-open, provided that its preimage
under $f$ is $\xi$-open,
\item \emph{hereditarily quotient} if whenever an open set $P$ includes
$f^{-}(y)$, then $f(P)$ is a neighborhood of $y$ (\footnote{for each $y\in\left|\tau\right|$.}),
\item \emph{countably biquotient }if, for each countable open cover $\mathcal{P}$
of $f^{-}(y)$, there exists a finite subfamily $\mathcal{P}_{0}$
such that $\bigcup_{P\in\mathcal{P}_{0}}f(P)$ is a neighborhood of
$y$ ($\footnotemark[21]$),
\item \emph{biquotient }if for each open cover $\mathcal{P}$ of $f^{-}(y)$,
there exists a finite subfamily $\mathcal{P}_{0}$ such that $\bigcup_{P\in\mathcal{P}_{0}}f(P)$
is a neighborhood of $y$ ($\footnotemark[21]$).
\end{enumerate}
Recall that a generalization of $H$-quotient maps to arbitrary (concrete)
reflectors resulted in\emph{ almost open} maps (with $H=\mathrm{I}$,
the identity reflector). On the other hand, in topology,\emph{ open
maps} are almost open. 

\section{Mixed properties and functorial inequalities}

Numerous classical properties of topologies can be expressed in terms
of inequalities, involving two or more functors, typically, a coreflector
and a reflector. These properties make sense for an arbitrary convergence,
and will be introduced here in the convergence framework. Let us say
that a convergence $\xi$ is $JE$ if
\begin{equation}
\xi\geq JE\xi,\label{JE}
\end{equation}
where $J$ is a (concrete) reflector and $E$ is a (concrete) coreflector
(See subsections \ref{subsec:Functors} and \ref{subsec:Corefelective-subclasses}).
A composition of two functors is, of course, a functor. 

Let us give a couple of examples of $JE$-properties. In the examples
below, we shall employ a reflector $\mathrm{S}_{1}$ on the class
of convergences that are determined by adherences pf countably based
filters (see Appendix \ref{subsec:Reflective-subclasses}), and a
coreflector $\mathrm{I}_{1}$ on the subclass of convergences of\emph{
countable character} (see Appendix \ref{subsec:Corefelective-subclasses}).
\begin{example}
(\footnote{Mind that \emph{sequential topology} and \emph{sequentially based
convergence} are different notions!}) A topology is called \emph{sequential} if each sequentially closed
set is closed. If $\xi$ is a topology, then $\mathrm{Seq\,}\xi$
is the coarsest \emph{sequentially based convergence} (see Appendix
\ref{subsec:Corefelective-subclasses}) in general non-topological,
that is finer than $\xi$. 

Then $\mathrm{T\,Seq}\,\xi$ stands for the topology, for which the
open sets and the closed sets are determined by sequential filters,
that is, are sequentially open and closed, respectively. Therefore,
a topology $\xi$ is sequential if it coincides with $\mathrm{T\,Seq}\,\xi$,
which is equivalent to
\begin{equation}
\xi\geq\mathrm{T\,Seq}\,\xi.\label{eq:TSeq}
\end{equation}
Incidentally, $\mathrm{T}\,\mathrm{Seq}=\mathrm{T}\,\mathrm{I}_{1}$,
hence the inequality above amounts to $\xi\geq\mathrm{T\,I}_{1}\,\xi$.
\end{example}

\begin{example}
A convergence $\xi$ is called \emph{Fréchet} if $x\in\ad_{\xi}A$
implies the existence of a sequential filter $\mathcal{E}$ such that
$A\in\mathcal{E}$ and $x\in\lm_{\xi}\mathcal{E}$. It is straightforward
that $\xi$ is Fréchet whenever
\[
\xi\geq\mathrm{S_{0}\,Seq}\,\xi=\mathrm{S}_{0}\,\mathrm{I}_{1}\,\xi.
\]
\end{example}

\begin{example}
A convergence $\xi$ is called \emph{countably bisequential} (\footnote{Also called \emph{strongly Fréchet.}})
if $x\in\bigcap_{n}\ad_{\xi}A_{n}$ for a decreasing sequence $(A_{n})_{n}$
implies the existence of a sequence $(x_{n})_{n}$ such that $x_{n}\in A_{n}$
for each $n$, and $x\in\lm_{\xi}(x_{n})_{n}$. It is straightforward
that $\xi$ is countably bisequential if and only if $\mathcal{A}$
is a countably based filter, and $x\in\ad_{\xi}\mathcal{A}$, then
there is a countably based filter $\mathcal{E}$ such that $\mathcal{A}\#\mathcal{E}$
and $x\in\lm\mathcal{E}$, that is, whenever
\[
\xi\geq\mathrm{S_{1}\,Seq}\,\xi=\mathrm{S}_{1}\,\mathrm{I}_{1}\,\xi.
\]
\end{example}

\begin{example}
A convergence $\xi$ is called \emph{bisequential} if for every ultrafilter
$\mathcal{U}$ such that $x\in\lm_{\xi}\mathcal{U}$, there exists
a countably based filter $\mathcal{H\leq\mathcal{U}}$ such that $x\in\lm_{\xi}\mathcal{H}$.
In other words, $\xi$ is bisequential whenever $\lm_{\xi}\mathcal{U}\subset\lm_{\mathrm{I_{1}}\xi}\mathcal{U}$
for each ultrafilter $\mathcal{U}$, that is, whenever
\[
\xi\geq\mathrm{S}\,\mathrm{I}_{1}\,\xi.
\]
\end{example}

It has been long well-known that (topologically) continuous quotient
maps preserve (topologically) sequential topologies, but do not preserve
Fréchet topologies. On the other hand, continuous hereditarily quotient
maps preserve Fréchet topologies, but not countably bisequential topologies,
which are preserved by continuous countably biquotient maps, and so
on.

Theorem \ref{thm:preserv} below explains the mechanism of preservation
of mixed properties. Roughly speaking, a $JE$-property is preserved
by continuous $L$-quotient maps provided that $L$ is a concrete
reflector such that $J\leq L$.

It follows directly from the definition that
\begin{prop}
\label{prop:JE} A convergence $\xi$ is $JE$ if and only if the
identity from $E\xi$ to $\xi$ is $J$-quotient.
\end{prop}

On the other hand, the following observation generalizes several classical
theorems concerning particular preservation cases by \noun{E. Michael}
\cite{quest}, \noun{A. V. Arhangel'skii} \cite{arh.factor}, \noun{V.
I. Ponomarev} \cite{Pon}, \noun{S. Hanai} \cite{Han}, \noun{F. Siwiec}
\cite{Siwiec}, and others.
\begin{thm}
\label{thm:preserv}If $f\in C(\xi,\tau)$ is a $J$-quotient map,
and $\xi$ is a $JE$-convergence, then $\tau$ is also a $JE$-convergence.
Conversely, each $JE$-convergence is a continuous $J$-quotient image
of an $E$-convergence.
\end{thm}

\begin{proof}
By definition, $f$ is a $J$-quotient map whenever $\tau\geq J(f\xi)$,
and $\xi$ is a $JE$-map whenever $\xi\geq JE\xi$. Putting these
two inequalities together, yields
\[
\tau\geq J(f\xi)\geq J(f(JE\xi)),
\]
and, by (\ref{eq:functorial}), $J(f(JE\xi))\geq JE(f\xi)$, which
is greater than $JE\tau$, in view of continuity. By concatenating
these formulas, we get $\tau\geq JE\tau$.

Conversely, $\tau\geq JE\tau=J(i(E\tau))$, where $i\in C(\xi,\tau)$
is the identity (see Proposition \ref{prop:JE}).
\end{proof}
The mentioned classical theorems, however, assert that each $JE$
topology is a $J$-quotient image of a topology (sometimes even metrizable
or paracompact) $\theta$ such that $\theta\geq E\theta$. In our
proof, $E\tau$ is in general not a topology. This topologicity requirement
might provide technical difficulties (see \cite[Section 5]{quest2},
for example).

Below, we shall give an example of the discussed preservation scheme
only in case of $E=\mathrm{I}_{1}$ (the coreflector of \emph{countable
character}). The interested readers will find similar, but more exhaustive,
tables in \cite[p. 11]{quest2} and in \cite[p. 400]{CFT}. The first
column indicates the type of quotient, the second the corresponding
reflector. A quotient type in a given line preserves the properties
of the same line and of the lines that are below. The strength of
the listed properties decreases downwards.\medskip{}

\begin{center}
\begin{tabular}{|c|c|c|c|}
\hline 
quotient type & reflector $J$ & mixed property & type $J\mathrm{I}_{1}$\tabularnewline
\hline 
\hline 
almost open & $\mathrm{I}$ & countable character & $\mathrm{I}_{1}$\tabularnewline
\hline 
biquotient & $\mathrm{S}$ & bisequential & $\mathrm{SI}_{1}$\tabularnewline
\hline 
countably biquotient & $\mathrm{S}_{1}$ & countably bisequential & $\mathrm{S}_{1}\mathrm{I}_{1}$\tabularnewline
\hline 
hereditarily quotient & $\mathrm{S}_{0}$ & Fréchet & $\mathrm{S}_{0}\mathrm{I}_{1}$\tabularnewline
\hline 
quotient & $\mathrm{T}$ & sequential & $\mathrm{T}\mathrm{I}_{1}$\tabularnewline
\hline 
\end{tabular}
\par\end{center}

\section{Various classes of perfect-like maps}

As for quotient-like maps, traditional definitions of perfect-like
maps incorporate continuity. In most situations, however, continuity
is superfluous, and no continuity is assumed here. Recall the definition
of $\mathbb{G}$-perfect maps: a map $f$ from $\left|\xi\right|$
to $\left|\tau\right|$ whenever (\ref{eq:per})
\[
\ad_{\tau}f[\mathcal{G}]\subset f(\ad_{\xi}\mathcal{G})
\]
holds for each filter $\mathcal{G}\in\mathbb{G}$. This can be also
reformulated as
\begin{equation}
y\in\ad_{\tau}f[\mathcal{G}]\Longrightarrow f^{-}(y)\cap\ad_{\xi}\mathcal{G}\neq\varnothing\label{eq:perfect}
\end{equation}
for each filter $\mathcal{G}\in\mathbb{G}$.

I list below the names of surjective maps $f$ corresponding to (\ref{eq:perfect})
for specific classes $\mathbb{G}$ of filters (written in brackets):
\begin{enumerate}
\item \emph{closed} (principal filters of closed sets),
\item \emph{adherent} (principal filters),
\item \emph{countably perfect} (countably based filters),
\item \emph{perfect} (all filters).
\end{enumerate}
\emph{Perfect} maps are \emph{countably perfect}, which are \emph{adherent},
which, in turn, are \emph{closed} (\footnote{To see that \emph{adherent} implies \emph{closed}, take a $\xi$-closed
set $G$, that is, $\ad_{\xi}G=G$ and set $\mathcal{G}=\{G\}^{\uparrow}$
in the definition of adherent map above, with $\mathbb{G}$ the class
of principal filters, and get $\ad_{\tau}f[G]\subset f(G)$, that
is, $f(G)$ is $\tau$-closed.}). \emph{Adherent maps} were introduced in \cite{quest2} in the context
of general convergences; they were not detected before, because they
coincide with \emph{closed maps} in the framework of topologies. Here
we show that they do not coincide in general.
\begin{example}
Let $\xi=\mathrm{S}_{0}\xi>\mathrm{T}\xi=\tau$. Then the identity
$i\in C(\xi,\tau)$ is a quotient map, that is, $\tau\geq\mathrm{T}\xi$,
hence a closed map by Remark \ref{rem:Iper-quot}. As $\mathrm{S}_{0}\xi>\tau$,
it is not a hereditarily quotient map, hence not an adherent map by
Remark \ref{rem:Iper-quot}.
\end{example}

A surjective map is \emph{closed} if it maps closed sets onto closed
sets. Indeed, if $\mathcal{G}=\{\cl_{\xi}G\}^{\uparrow}$ then $\ad_{\xi}G\subset G$,
hence (\ref{eq:per}) becomes $\ad_{\tau}f(G)\subset f(G)$, that
is, whenever$f(G)$ is $\tau$-closed.

The following two propositions characterize perfect and countably
perfect maps in topological spaces in traditional terms. They are
particular cases of a more abstract situation, discussed in \cite[Proposition XV.3.6]{CFT}.

A map between topological spaces is called \emph{perfect} \cite[p. 236]{Eng}
if it is continuous, closed and all its fibers are compact (in a Hausdorff
domain of the map \footnote{Hausdorffness is assumed, because compactness in \cite{Eng} is defined
only for Hausdorff topologies.}). In this paper we consider only surjective maps. On the other hand,
continuity is irrelevant in these propositions.
\begin{prop}
\label{prop:perf}A surjective map between topological spaces is \emph{perfect}
if and only if it is closed and all its fibers are compact.
\end{prop}

\begin{prop}
\label{prop:count_perf}A surjective map between topological spaces
is \emph{countably perfect} if and only if it is closed and all its
fibers are countably compact.
\end{prop}

We shall provide proofs to these propositions in Section \ref{sec:Compactness-characterizations}.

\begin{onehalfspace}
The following table details Proposition \ref{prop:transferable}.
The strength of the listed properties decreases downwards.
\end{onehalfspace}
\begin{onehalfspace}
\begin{center}
\begin{tabular}{|c|c|c|}
\hline 
perfect-like & $\Rightarrow$ & quotient-like\tabularnewline
\hline 
\hline 
 &  & open\tabularnewline
\hline 
 &  & almost open\tabularnewline
\hline 
perfect & $\Rightarrow$ & biquotient\tabularnewline
\hline 
countably perfect & $\Rightarrow$ & countably biquotient\tabularnewline
\hline 
adherent & $\Rightarrow$ & hereditarily quotient\tabularnewline
\hline 
closed & $\Rightarrow$ & topologically quotient\tabularnewline
\hline 
\end{tabular}\medskip{}
\par\end{center}
\end{onehalfspace}

\begin{onehalfspace}
No arrow can be reversed. Indeed,
\end{onehalfspace}
\begin{example}
\label{exa:classic}Let $f:\mathbb{R}\rightarrow S^{1}$ be given
by $f(x):=(\cos2\pi x,\sin2\pi x)$ (\footnote{We consider the usual topologies on the real line $\mathbb{R}$ and
on the unit circle $S^{1}$. }). It follows immediately from Proposition \ref{prop:open} that $f$
is open, hence, has all the properties from the right-hand column.
Notice that the set $\{n+\frac{1}{n}:n\in\mathbb{N}\}$ is closed,
but its image under $f$ is not closed, so that $f$ is not closed,
and thus has no property from the left-hand column.
\end{example}

\section{Compactness characterizations\label{sec:Compactness-characterizations}}

In Appendix \ref{subsec:Compactness}, a relation $R\subset W\times Z$
(with convergences $\theta$ on $W$ and $\sigma$ on $Z$) is said
to be $\mathbb{J}$\emph{-compact }at $w$\emph{ }if for every $\mathcal{F}$
such that $w\in\lm_{\theta}\mathcal{F}$ implies that $R(w)\#\ad_{\sigma}\mathcal{J}$
for each $\mathcal{J}\#R[\mathcal{F}]$ such that $\mathcal{J}\in\mathbb{J}$.
Now I will characterize the classes of maps discussed in terms of
such relations. These characterizations enable us to see new interrelations
(see Propositions \ref{prop:perf} and \ref{prop:count_perf} and
Section \ref{sec:A-couple-of}).

\subsection{Continuous maps}
\begin{thm}
If $J$ is a $\mathbb{J}$-adherence-determined reflector, then $f\in C(J\xi,J\tau)$
if and only if $f$ is $\mathbb{J}$-compact from $\xi$ to $\tau$.
\end{thm}

\begin{proof}
By definition,
\begin{equation}
\lm_{J\tau}f[\mathcal{F}]=\bigcap_{\mathbb{J}\ni\mathcal{H}\#f[\mathcal{F}]}\ad_{\tau}\mathfrak{\mathcal{H}}.\label{eq:adhoc}
\end{equation}
Let $x\in\lm_{\xi}\mathcal{F}$. If $f\in C(J\xi,J\tau)=C(\xi,J\tau)$
(\footnote{$C(\xi,J\tau)\subset C(J\xi,J\tau)$, because $\xi\geq J\xi$. On
the other hand, by applying $J$ to $f\xi\geq J\tau$, we get $J(f\xi)\geq JJ\tau=J\tau.$
As $J$ is a functor, $f(J\xi)\geq J(f\xi)$ by (\ref{eq:funct}).}), then $f(x)$ belongs to (\ref{eq:adhoc}), equivalently, $\{f(x)\}\#\ad_{\tau}\mathfrak{\mathcal{H}}$
for each $\mathcal{H}\in\mathbb{J}$ such that $\mathcal{H}\#f[\mathcal{F}]$,
that is, $f$ is $\mathbb{J}$-compact from $\xi$ to $\tau$. 
\end{proof}

\subsection{Perfect-like maps}

Perfect-like maps $f$ were characterized in \cite{D.comp,D.comp_err}
in terms of various types of compactness of $f^{-}$.
\begin{thm}
\label{thm:perfect=00003Dcomp}A surjective map $f$ is $\mathbb{J}$-perfect
if and only if the relation $f^{-}$ is $\mathbb{J}$-compact.
\end{thm}

\begin{proof}
Recall (\ref{eq:perf_form}) that a map $f:\left|\xi\right|\rightarrow\left|\tau\right|$
is $\mathbb{J}$-perfect whenever, for each $\mathcal{J}\in\mathbb{J}$
and each $y\in\left|\tau\right|$,
\[
y\in\ad_{\tau}f[\mathcal{J}]\Longrightarrow f^{-}(y)\#\ad_{\xi}\mathcal{J}.
\]
Assume that $f$ is $\mathbb{J}$-perfect and let $y\in\lm_{\tau}\mathcal{F}$
and consider $\mathcal{J}\#f^{-}[\mathcal{F}]$ such that $\mathcal{J}\in\mathbb{J}$.
Accordingly, $f[\mathcal{J}]\#\mathcal{F}$ hence $y\in\ad_{\tau}f[\mathcal{J}]$,
and thus, by assumption, $f^{-}(y)\#\ad_{\xi}\mathcal{J}$, which
proves that $f^{-}$ is $\mathbb{J}$-compact. 

Conversely, let $f^{-}$ be $\mathbb{J}$-compact, and $y\in\ad_{\tau}f[\mathcal{J}]$
with $\mathcal{J}\in\mathbb{J}$. Consequently, there exists $\mathcal{F}\#f[\mathcal{J}]$
and such that $y\in\lm_{\tau}\mathcal{F}$. By the compactness assumption,
$f^{-}(y)\#\ad_{\xi}\mathcal{J}$ for each $\mathcal{J}\in\mathbb{J}$
such that $\mathcal{J}\#f^{-}[\mathcal{F}]$.
\end{proof}
By Proposition \ref{prop:image_comp}, 
\begin{prop}
Let $\mathbb{J}$ be transferable, and $f$ be $\mathbb{J}$-perfect.
Then $f^{-}(K)$ is $\mathbb{J}$-compact, provided that $K$ is $\mathbb{J}$-compact.
\end{prop}

\begin{proof}
The notion of graph-closedness (Definition \ref{def:graph-closed})
will be used in the following 
\end{proof}
\begin{cor}
If $J$ is $\mathbb{J}$-adherence-determined, then a surjective map
$f$ (from $\left|\xi\right|$ to $\left|\tau\right|$) is $\mathbb{J}$\emph{-perfect}
provided that $f$ is graph-closed and
\begin{equation}
f^{-}\tau\geq J(\chi_{\xi}),\label{eq:perf_abstract}
\end{equation}
where $\chi_{\xi}$ is the characteristic convergence of $\xi$ (\ref{eq:char_def}).
\end{cor}

We are now in a position to easily prove Propositions \ref{prop:perf}
and \ref{prop:count_perf}.
\begin{proof}[Proof of Proposition \ref{prop:perf}]
 Let $f\in C(\xi,\tau)$ be perfect. Then the inverse image $f^{-}[\{y\}^{\uparrow}]$
is $\xi$-compact at $f^{-}(y)$ (Definition \ref{def:compact_at})
for each $y\in\left|\tau\right|$, because singletons are compact.
This means that $f^{-}(y)$ is a $\xi$-compact set. On the other
hand, since (\ref{eq:perfect}) holds, in particular, for principal
filters $\mathcal{G}$ of $\xi$-closed sets, $f$ is closed.

To show the converse, let $\mathcal{G}\#f^{-}[\mathcal{F}]$, equivalently,
$G\#f^{-}[\mathcal{F}]$ for each $G\in\mathcal{G}$, where $y\in\lm_{\xi}\mathcal{F}$.
As $f$ is a closed map, hence, in particular, $f^{-}$ is an $\mathbb{F}_{0}$-compact
relation, $\cl{}_{\xi}G\cap f^{-}(y)\neq\varnothing$ for every $G\in\mathcal{G}$.
Because $f^{-}(y)$ is assumed to be $\xi$-compact,
\[
\ad_{\ensuremath{\xi}}\mathcal{G}\cap f^{-}(y)=\bigcap\nolimits _{G\in\mathcal{G}}\cl_{\xi}G\cap f^{-}(y)\neq\varnothing,
\]
 which concludes the proof.
\end{proof}
Proposition \ref{prop:count_perf} can be demonstrated analogously
\emph{mutatis mutandis.} 

\subsection{Quotient-like maps}

It turns out that also quotient-like maps admit characterizations
in terms of various types of compactness, as was shown in \cite[Theorem 3.8]{myn.relations}
by \noun{F. Mynard}.
\begin{thm}
\label{thm:quot=00003Dcompactoid}A surjective map $f$ is $\mathbb{J}$-quotient
from $\xi$ to $\tau$ if and only if $f$ is $\mathbb{J}$-compact
from $f^{-}\tau$ to $f\xi$.
\end{thm}

\begin{proof}
($\Rightarrow$) If $x\in\lm_{f^{-}\tau}\mathcal{F}$, then
\[
f(x)\in f(\lm_{f^{-}\tau}\mathcal{F})=f(f^{-}(\lm_{\tau}f[\mathcal{F}]))=\lm_{\tau}f[\mathcal{F}],
\]
because $f$ is surjective. Consider $\mathcal{J}\#f[\mathcal{F}]$
such that $\mathcal{J}\in\mathbb{J}$. By (\ref{eq:quot}), $f^{-}(f(x))\#\ad_{\xi}f^{-}[\mathcal{J}]$,
hence
\[
f(x)\in f(\ad_{\xi}f[\mathcal{J}])=\ad_{f\xi}\mathcal{J}.
\]
Because $f^{-}[\mathcal{J}]\#\mathcal{F}$, we infer that $f(x)\in\ad_{f\xi}\mathcal{J}$,
so that $f[\mathcal{F}]$ is $\mathbb{J}$-compact at $\{f(x)\}$
in $f\xi$.

($\Leftarrow$) Let $\mathcal{J}\in\mathbb{J}$ and $y\in\ad_{\tau}\mathcal{J}$.
Hence, taking into account the equality $f[f^{-}[\mathcal{J}]]=\mathcal{J}$,
\[
f^{-}(y)\subset f^{-}(\ad_{\tau}\mathcal{J})=f^{-}(\ad_{\tau}f[f^{-}[\mathcal{J}]]),
\]
we infer that $f^{-}(y)\subset\ad_{f^{-}\tau}f^{-}[\mathcal{J}]$.
Let $x\in f^{-}(y),x\in\lm_{f^{-}\tau}\mathcal{F}$ and $\mathcal{F}\#f^{-}[\mathcal{J}]$,
equivalently $f[\mathcal{F}]\#\mathcal{J}$. As $f$ is $\mathbb{J}$-compact
from $f^{-}\tau$ to $f\xi$, the filter-base $f[\mathcal{F}]$ is
$\mathcal{\mathbb{J}}$-compact with respect to $f\xi$. In particular,
$\ad_{f\xi}\mathcal{J}\#\{f(x)\}$, that is, $y\in\ad_{f\xi}\mathcal{J}$,
so that, (\ref{eq:quot}) is proved.
\end{proof}
We recall that the topologicity of $\xi$ does not entail that of
$f\xi$ (\footnote{Any finitely deep convergence can be represented as the quotient of
a topology.}), so that Theorem \ref{thm:quot=00003Dcompactoid} cannot be addressed
in the topological framework. 

\section{\label{sec:A-couple-of}Preservation of completeness}

It is well-known that (in topological spaces) continuous maps preserve
\emph{compactness}, but neither \emph{local compactness,} nor \emph{topological
completeness}, which are preserved by continuous perfect maps.
\begin{defn*}
We say that a filter is \emph{adherent} if its adherence is non-empty;
\emph{non-adherent} if its adherence is empty.
\end{defn*}
Our representation theorems for perfect maps enable us to perceive
the reason for the facts mentioned above (\footnote{To establish $\geq$ in (\ref{compl}) we need that the map be perfect,
to assure that the images of non-adherent filters be non-adherent.
This is not needed if $\xi$ is compact, because then all filters
on $\left|\xi\right|$ are $\xi$-adherent. }). I hope that you would appreciate a remarkable simplicity of the
proof below. The argument uses just two facts:
\begin{enumerate}
\item the image of each adherent filter under a continuous map is adherent
(\footnote{Equivalently, the preimage of a non-adherent filter under a continuous
is non-adherent.}), and
\item the preimage of each adherent filter under a perfect map is adherent
(\footnote{Equivalently, the image of a non-adherent filter under a perfect map
is non-adherent.}).
\end{enumerate}
Recall (Appendix \ref{subsec:Completeness}) that the completeness
number $\mathrm{compl}(\theta)$ of a convergence $\theta$ is the
least cardinal $\kappa$, for which $\theta$ is $\kappa$\nobreakdash-complete. 
\begin{thm}[{\cite[XI.6]{CFT}}]
 If $\xi$ and $\tau$ are convergences, and if $f\in C(\xi,\tau)$
is a (surjective) perfect map, then
\begin{equation}
\mathrm{compl}(\xi)=\mathrm{compl}(\tau).\label{compl}
\end{equation}
\end{thm}

\begin{proof}
Let $\mathcal{\mathbb{J}}$ be a a collection of non-$\xi$\nobreakdash-adherent
filters, with $\mathrm{card}(\mathbb{J})\leq\kappa$, such that a
filter $\mathcal{G}$ is $\xi$\nobreakdash-adherent, provided that
$\neg(\mathcal{G}\#\mathcal{J})$ for each $\mathcal{J}\in\mathbb{J}$.
As $f$ is perfect, $f[\mathcal{J}]$ is non-$\tau$\nobreakdash-adherent
for each $\mathcal{J}\in\mathbb{J}$, so that $\mathrm{card}(\{f[\mathcal{J}]:\mathcal{J}\in\mathbb{J}\})\leq\kappa$.
If now $\neg(\mathcal{L}\#f[\mathcal{J}])$ then $\neg(f^{-}[\mathcal{L}]\#\mathcal{J})$,
and by completeness assumption above, $\ad_{\xi}f^{-}[\mathcal{L}]\neq\varnothing$.
As $f$ is continuous, $\varnothing\neq\ad_{\tau}f[f^{-}[\mathcal{L}]]$,
and $f[f^{-}[\mathcal{L}]]=\mathcal{L}$, because $f$ is surjective.
We conclude that $\mathrm{compl}(\xi)\geq\mathrm{compl}(\tau)$.

Conversely, let $\mathbb{M}$ be a collection of non-$\tau$\nobreakdash-adherent
filters, with $\mathrm{card}(\mathbb{M})\leq\kappa$, such that a
filter $\mathcal{G}$ is $\tau$\nobreakdash-adherent, provided that
$\neg(\mathcal{G}\#\mathcal{M})$ for each $\mathcal{M}\in\mathbb{M}$.
As $f$ is continuous, $f^{-}[\mathcal{M}]$ is non-$\xi$\nobreakdash-adherent
for each $\mathcal{M}\in\mathbb{M}$, and $\mathrm{card}(\{f^{-}[\mathcal{M}]:\mathcal{M}\in\mathbb{M}\})\leq\kappa$. 

If now $\neg(\mathcal{G}\#f^{-}[\mathcal{M}])$ then $\neg(f[\mathcal{G}]\#\mathcal{M})$,
thus, by the completeness assumption above, $\varnothing\neq\ad_{\tau}f[\mathcal{G}]$,
hence there is a filter $\mathcal{F}$ such that $\varnothing\neq\lm_{\tau}\mathcal{F}$
and $\mathcal{F}\#f[\mathcal{G}]$, equivalently $f^{-}[\mathcal{F}]\#\mathcal{G}$.
As $f$ is perfect, $f^{-}[\mathcal{F}]$ is $\xi$\nobreakdash-compactoid,
and thus $\ad_{\xi}\mathcal{G}\neq\varnothing$. As a result, $\mathrm{compl}(\tau)\geq\mathrm{compl}(\xi)$.
\end{proof}
We infer that 
\begin{cor}
If $f\in C(\xi,\tau)$ is a (surjective) perfect map, then $\xi$
is compact (resp. locally compact, topologically complete) if and
only if $\tau$ is compact (resp. locally compact, topologically complete)\emph{.}
\end{cor}

\section{Appendix}

\subsection{Relations\label{subsec:Relations}}

It is well-known that maps and their converses are special cases of
binary relations; equivalence and order are other familiar examples
of such relations. A purpose of this section is not that of recalling
basic facts about relations, but rather to fix a notation and present
a\emph{ calculus of relations}, which, though elementary, is very
convenient.

The framework of relations is more flexible than that of maps, and
their calculus simplifies and clarifies reasoning. If $R\subset W\times Z$
is a relation, then $Rw:=\{z\in Z:(w,z)\in R\}$, and $R^{-}z:=\{w\in W:(w,z)\in R\}$,
which defines the \emph{inverse relation} $R^{-}$ of $R$. For $A\subset W,B\subset Z,$
we denote the \emph{image} of $A$ under $R$ and the \emph{preimage}
of $B$ under $R$, respectively, by
\[
RA:=\bigcup\nolimits _{w\in A}Rw,\textrm{ and }R^{-}B:=\bigcup\nolimits _{z\in B}R^{-}z.
\]

Notice that, for a relation $R\subset W\times Z,A\subset W$ and $B\subset Z$,
the following expressions are equivalent:
\begin{gather*}
RA\cap B\neq\textrm{Ø,}\\
A\cap R^{-}B\neq\textrm{Ø,}\\
(A\times B)\cap R\neq\textrm{Ø.}
\end{gather*}

If $X$ is a set, the relation of \emph{grill} $\#$ on $2^{X}:=\{A:A\subset X\}$
is defined by
\[
A\#B\Longleftrightarrow A\cap B\neq\varnothing.
\]
Now, if $R\subset W\times Z$ is a relation, then the grill interplays
with $R$ and $R^{-}$. It can be used to express the equivalences
above:
\[
RA\#B\Longleftrightarrow A\#R^{-}B\Longleftrightarrow(A\times B)\#R.
\]
The relation of grill is extended to families of subsets, if $\mathcal{A}$
and $\mathcal{B}$ are families of subsets (of a given set), then
\[
\mathcal{A}\#\mathcal{B}\Longleftrightarrow\underset{A\in\mathcal{A}}{\forall}\:\underset{B\in\mathcal{B}}{\forall}\,A\#B.
\]
If $R\subset W\times Z,\mathcal{A}\subset2^{W}$ and $\mathcal{B}\subset2^{Z}$,
then $R[\mathcal{A}]:=\{RA:A\in\mathcal{A}\}$ and $R^{-}[\mathcal{B}]:=\{R^{-}B:B\in\mathcal{B}\}$.
Notice that
\begin{equation}
R[\mathcal{A}]\#\mathcal{B}\Longleftrightarrow\mathcal{A}\#R^{-}[\mathcal{B}].\label{eq:grill_rel}
\end{equation}

\begin{defn}
A \emph{relation} $R\subset W\times Z$ is said to be \emph{injective}
if $Rw_{0}\cap Rw_{1}\neq\varnothing$ implies $w_{0}=w_{1}$; it
is called \emph{surjective} if $RW=Z$ (\cite{Analyse}\cite{CFT}). 
\end{defn}

In particular, a map $f\subset X\times Y$ is injective if $f(x_{0})=f(x_{1})$
implies $x_{0}=x_{1}$, and it is surjective if $f^{-}(y)\neq\varnothing$
for each $y\in Y$. It is straightforward that
\begin{prop}
\label{prop:rel-map} A relation $R\subset X\times Y$ is a map from
$X$ to $Y$ if and only if the inverse relation $R^{-}$ is injective
and surjective \emph{(}\footnote{Indeed, $R$ is surjective, that is, $R^{-}Y=X$ whenever for each
$x\in X$ there exists $y\in Rx$. On the other hand, $R^{-}$ is
injective if for each $x\in X$ there is at most one $y\in Rx$.

If one of the involved sets $X,Y$ is empty, then the product $X\times Y=\varnothing,$
hence there is only one relation $R=\varnothing$. There are two cases:
if $X=\varnothing$, then $R$ is a map; if $Y=\varnothing$ and $X\neq\varnothing$,
then R is not a map \cite[p. 506]{CFT}.}\emph{)}.
\end{prop}

\begin{defn}
\label{def:graph-closed} Let $\theta$ and $\sigma$ be convergences
on $W$ and $Z,$ respectively. A relation $R\subset W\times Z$ is
called \emph{graph-closed at $w$} if $\ad_{\sigma}R[\mathcal{F}]\subset Rw$
for each filter $\mathcal{F}$ such that $w\in\lm_{\theta}\mathcal{F}$.

Notice that a relation $R$ is graph-closed at each point if and only
if $R$ is a closed subset of the product. In fact, $(w,z)\in\ad_{\theta\times\sigma}R$
whenever there exist $\mathcal{F}$ and $\mathcal{G}$ such that $R\#(\mathcal{F}\times\mathcal{G}),w\in\lm_{\theta}\mathcal{F}$
and $z\in\lm_{\sigma}\mathcal{G}$. By (\ref{eq:grill_rel}), $R\#(\mathcal{F}\times\mathcal{G})$
is equivalent to $\mathcal{G}\#R[\mathcal{F}]$, so that $z\in\lm_{\sigma}\mathcal{G}\subset\ad_{\sigma}R[\mathcal{F}]$
. On the other hand, $(w,z)\in R$ if and only if $z\in Rw$.

It follows that $R$ is graph-closed at every point if and only if
$R^{-}$ is graph-closed at every point.

Finally, notice that if a map $f$ is continuous and valued in a Hausdorff
convergence (\footnote{A convergence $\theta$ is said to be \emph{Hausdorff} if $\{y_{0},y_{1}\}\subset\lm_{\theta}\mathcal{F}$
implies that $y_{0}=y_{1}$. }), then the relation $f$ is graph-closed, hence also $f^{-}$ is
graph-closed. Indeed, if $f\in C(\xi,\tau)$ and $x\in\lm_{\xi}\mathcal{F}$
then $\{f(x)\}=\lm_{\tau}f[\mathcal{F}]=\ad_{\tau}f[\mathcal{F}]$,
both equalities are a consequence of the fact that $\tau$ is Hausdorff.
\end{defn}

\subsection{Filters\label{subsec:Filters}}

A non-empty family $\mathcal{F}$ of subsets of a set $X$ is called
a \emph{filter} whenever $F_{0}\in\mathcal{F}$ and $F_{1}\in\mathcal{F}$
if and only if $F_{0}\cap F_{1}\in\mathcal{F}$. A filter $\mathcal{F}$
is called \emph{non-degenerate} if $\varnothing\notin\mathcal{F}$. 

Let $\overline{\mathbb{F}}X$ stand for the set of all, and by $\mathbb{F}X$
of all non-degenerate filters on $X$. If $\mathcal{F}$ is a filter,
then $\mathcal{B}\subset\mathcal{F}$ is called a \emph{base} of $\mathcal{F}$
(or $\mathcal{F}$ is said to be \emph{generated} by $\mathcal{B}$)
if for each $F\in\mathcal{F}$ there is $B\in\mathcal{B}$ such that
$B\subset F$. 

A filter is called \emph{countably based} if it admits a countable
base. We write $\mathbb{F}_{1}X$ for the set of \emph{countably based}
\emph{filters} on $X$.

A filter is called \emph{principal} if it admits a finite base, equivalently,
a one element base. The principal filters on a set $X$ are of the
form $\{A\}^{\uparrow}:=\{F\subset X:F\supset A\}$. A filter $\mathcal{F}$
is said to be \emph{free }whenever $\bigcap_{F\in\mathcal{F}}F=\varnothing\text{.}$

Ordered by inclusion, $\overline{\mathbb{F}}X$ is a complete lattice,
that is, $\bigcap\mathbb{G}$ and $(\mathbb{\bigcup\mathbb{G}})^{\cap}$
(\footnote{$\mathcal{A}^{\cap}$ stands for the set of all the finite intersections
of the elements of $\mathcal{A}$.}) is a (possibly degenerate) filter for each $\varnothing\neq\mathbb{G}\subset\overline{\mathbb{F}}X$.
The least element of $\overline{\mathbb{F}}X$ is $\{X\}$, and the
greatest one is $2^{X}$, the family of all subsets of $X$ (\footnote{Which is the only degenerate filter on $X$.}).
Restricted to $\mathbb{F}X=\overline{\mathbb{F}}X\setminus\{2^{X}\}$,
this order admits arbitrary infima and $\bigwedge\mathbb{G}=\bigcap\mathbb{G}$;
however, $\bigvee\mathbb{G}$ exists in $\mathbb{F}X$ if and only
if $(\mathbb{\bigcup\mathbb{G}})^{\cap}\neq2^{X}$, that is, whenever
$G_{1}\cap G_{2}\cap\ldots\cap G_{n}\neq\varnothing$ for each finite
subfamily $\{G_{1},G_{2},\ldots,G_{n}\}$ of $\bigcup\mathbb{G}$,
in which case, $\bigvee\mathbb{G}=(\bigcup\mathbb{G})^{\cap}$.

The elements of $\mathbb{F}X$ that are maximal for the order defined
by inclusion are called\emph{ ultrafilters}. The set of ultrafilters
that include a filter $\mathcal{F}$ is denoted by $\beta\mathcal{F}$.
For each filter $\mathcal{F}$,
\[
\mathcal{F}=\bigcap_{\mathcal{U}\in\beta\mathcal{F}}\mathcal{U}.
\]

An ultrafilter is either free or principal; in the latter case, the
ultrafilter on a (non-empty) set $X$ is of the form $\{A\subset X:x\in A\}$
for some $x\in X$.

Let us recall a very useful link between a filter $\mathcal{F}$ and
the set $\beta\mathcal{F}$ of its ultrafilters \cite[Proposition II.6.5]{CFT}.
\begin{thm}[Compactness property of filters]
\label{thm:Comp_ultra} Let $\mathcal{F}$ be a filter, and for each
$\mathcal{U}\in\beta\mathcal{F}$ let $F_{\mathcal{U}}\in\mathcal{U}.$
Then there is a finite subset $\boldsymbol{F}$of $\beta\mathcal{F}$
such that $\bigcup_{\mathcal{U}\in\boldsymbol{F}}F_{\mathcal{U}}\in\mathcal{F}$.
\end{thm}

Every filter $\mathcal{F}$ on $X$ admits a unique decomposition
\[
\mathcal{F}=\mathcal{F}^{\ast}\wedge\mathcal{F}^{\bullet}\;\text{and}\;\mathcal{F}^{\ast}\vee\mathcal{F}^{\bullet}=2^{X},
\]
 to two (possibly degenerate) filters such that $\mathcal{F}^{\ast}$
is free, $\mathcal{F}^{\bullet}$ is principal. Indeed, if $\mathcal{F}\in\mathbb{F}X$,
then let $F_{\bullet}:=\bigcap_{F\in\mathcal{F}}F$ and $\mathcal{F}^{\bullet}:=\{A\subset X:F_{\bullet}\subset A\}$,
and $\mathcal{F^{\ast}}:=\{F\setminus F_{\bullet}:F\in\mathcal{F}\}^{\uparrow}$.
\begin{defn}
\label{def:cofinite}Let $A\subset B\subset X$. The family
\[
(\nicefrac{B}{A})_{0}:=\{F\subset X:A\subset F,\mathrm{card}(B\setminus F)<\infty\}
\]
is a filter, called a \emph{cofinite filter of} $B$ \emph{centered
at} $A$. 
\end{defn}

If\emph{ $A=\varnothing$ we drop ``centered at $A$'' and write
$(B)_{0}$. }If $B\setminus A$ is infinite, then the \emph{free part}
of $(\nicefrac{B}{A})_{0}$\emph{ is} $(B\setminus A)_{0}$, and the
\emph{principal part} is $\{A\}^{\uparrow}:=\{F\subset X:A\subset F\}$.

\begin{figure}[H]
\begin{centering}
\includegraphics[scale=0.35]{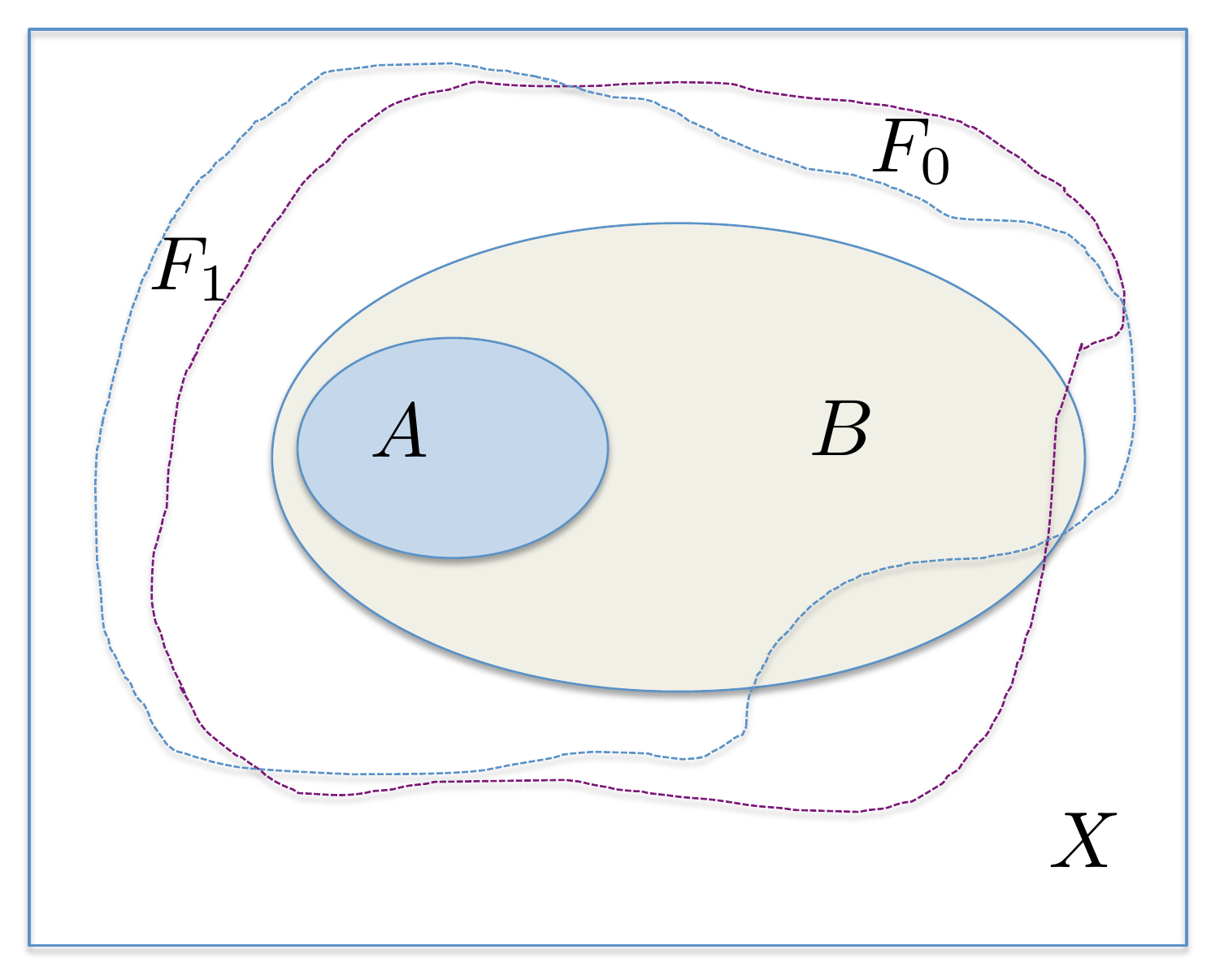}\caption{The \emph{cofinite filter} of $B$ centered at $A$ consists of all
those $F$ that include $A$ and such that $B\setminus F$ is finite.}
\par\end{centering}
\end{figure}

Observe that if $X$ is an infinite set, then the cofinite filter
$(X)_{0}$ of $X$ is the coarsest free filter on $X$. In particular,
$(X)_{0}$ is the infimum of all free ultrafilters on $X$.

\subsection{Sequential filters\label{subsec:Sequential-filters}}

A filter $\mathcal{E}$ on $X$ is called \emph{sequential }if there
exists a sequence $(x_{n})_{n}$ of elements of $X$ such that $\{\{x_{n}:n>m\}:m\in\mathbb{N}\}$
is a base of $\mathcal{E}$. Sequential filters are non-degenerate,
countably based and each contains a countable set, either infinite
or finite. Moreover
\begin{prop*}
A filter is sequential if and only if it is a cofinite filter of a
countable set $B$ centered at a subset $A$ of $B$.
\end{prop*}
\begin{example*}
The sequence $1,\frac{1}{2},\frac{1}{3},\ldots,\frac{1}{n},\ldots$
generates a cofinite filter of $\{\frac{1}{n}:n\in\mathbb{N}_{0}\}$
(\footnote{Where $\mathbb{N}_{0}:=\mathbb{N}\setminus\{0\}$.}) centered
at the empty set $\varnothing$; therefore it is a free filter. A
constant sequence with $x_{n}=1$ for each $n$, generates a cofinite
filter of $\{1\}$ at $\{1\}$; it is a principal filter. The sequence
$1,\frac{1}{2},\frac{1}{2},\frac{1}{3},\frac{1}{3},\frac{1}{3},\ldots,\underbrace{\tfrac{1}{n},\ldots,\tfrac{1}{n}}_{n\textrm{ times}}$
generates the same filter as the sequence $1,\frac{1}{2},\frac{1}{3},\ldots,\frac{1}{n},\ldots$.
Finally, the sequence $1,1,\frac{1}{2},1,\frac{1}{2},\frac{1}{3},\ldots,\underbrace{1,\ldots,\tfrac{1}{n}}_{n\textrm{ terms}}$
is a cofinite filter of $\{\frac{1}{n}:n\in\mathbb{N}_{0}\}$ centered
at $\{\frac{1}{n}:n\in\mathbb{N}_{0}\}$; it generates the principal
filter $\{F\subset X:$ $\{\frac{1}{n}:n\in\mathbb{N}_{0}\}\subset F\}$.
\end{example*}

\subsection{Adherence-determined subclasses\label{subsec:Reflective-subclasses}}

Reflective classes of topologies, pretopologies, paratopologies, and
pseudotopologies are \emph{adherence-determined}. 

If for each convergence $\theta$, there exists a collection of filters
$\mathbb{H}(\theta)$ on $\left|\theta\right|$, containing all principal
filters, and such that $\eta\leq\theta$ implies $\mathfrak{\mathbb{H}(\eta)\subset\mathbb{H}(\theta)},\mathbb{H}(H\theta)=\mathbb{H}(\theta)$,
and $\mathcal{H}\in\mathbb{H}(\tau)$ implies $f^{-}[\mathcal{H}]\in\mathbb{H}(f^{-}\tau)$
for all $\eta,\theta,\tau$ and $f$, then
\[
\lm_{H\theta}\mathcal{F}=\bigcap_{\mathbb{H}(\theta)\ni\mathcal{H}\#\mathcal{F}}\ad_{\theta}\mathcal{H}
\]
defines a concrete reflector $H$ \cite[Corollary XIV.3.3]{CFT}.
In Subsection \ref{subsec:Adherence-determined-reflectors}, we have
already discussed and applied adherence-determination by classs of
filters independent of the convergence. The class of topologies is
adherence-determined by the class of closed principal filters, hence,
depending on the convergence.

Here is a table of some important adherence-determined reflectors: 

\begin{table}[H]
\begin{tabular}{|c|c|c|c|}
\hline 
convergence class & reflector $H$ & $\mathbb{H}$ & filter class\tabularnewline
\hline 
\hline 
pseudotopologies & $\mathrm{S}$ & $\mathfrak{\mathbb{F}}$ & all filters\tabularnewline
\hline 
paratopologies & $\mathrm{S}_{1}$ & $\mathbb{F}_{1}$ & countably based filters\tabularnewline
\hline 
pretopologis & $\mathrm{S}_{0}$ & $\mathbb{F}_{0}$ & principal filters\tabularnewline
\hline 
topologies & $\mathrm{T}$ & $\mathbb{F}_{0}(\cdot)$ & closed principal filters\tabularnewline
\hline 
\end{tabular}
\end{table}

It is clear that
\[
\mathbb{F}_{0}(\cdot)\subset\mathbb{F}_{0}\subset\mathbb{F}_{1}\subset\mathbb{F},
\]
and thus,
\[
\mathrm{T}\leq\mathrm{S}_{0}\leq\mathrm{S}_{1}\leq\mathrm{S}.
\]

In particular, a convergence $\xi$ is a \emph{pretopolog}y whenever
(\footnote{This is equivalent to the fact that, for each $x$, there exists a
coarsest filter that $\xi$-converges to $x$. It is called the vicinity
filter of $\xi$ at $x$, and is denoted by $\mathcal{V}_{\xi}(x)$.})
\[
\lm_{\xi}\mathcal{F}=\bigcap\nolimits _{H\in\mathcal{F}^{\#}}\ad_{\xi}H,
\]
and $\xi$ is a \emph{pseudotopology} whenever
\[
\lm_{\xi}\mathcal{F}=\bigcap\nolimits _{\mathcal{H}\#\mathcal{F}}\ad_{\xi}\mathcal{H}=\bigcap\nolimits _{\mathcal{U}\in\mathcal{\beta F}}\ad_{\xi}\mathcal{U}=\bigcap\nolimits _{\mathcal{U}\in\mathcal{\beta F}}\lm_{\xi}\mathcal{U},
\]
where $\beta\mathcal{F}$ stands for the set of all ultrafilters that
are finer than $\mathcal{F}$. 
\begin{example}[Non-topological pretopology]
\label{exa:non-top} Let $(X_{n})_{n}$ be a sequence of disjoint
countably infinite sets, let $x_{n}\in X_{n}$ for each $n$, and
$x_{\infty}\notin X_{n}$ for each $n\in\mathbb{N}$. Consider the
set
\[
X:=\bigcup\nolimits _{n\in\mathbb{N}}X_{n}\cup\{x_{\infty}\}.
\]
Before continuing, recollect the notion of \emph{cofinite filter}
(Appendix, Definition \ref{def:cofinite}). Consider on $X$ a pretopology
$\theta$ given by setting $\mathcal{V_{\theta}}(x_{n})$ to be the
cofinite filter of $X_{n}$ centered in $x_{n}$, and $\mathcal{V_{\theta}}(x_{\infty})$
to be the cofinite filter of $X_{\infty}:=\{x_{\infty},x_{0},\ldots\}$
centered at $x_{\infty}$, while all the other points be isolated.
\begin{example}
\begin{figure}[H]
\begin{centering}
\includegraphics[scale=0.4]{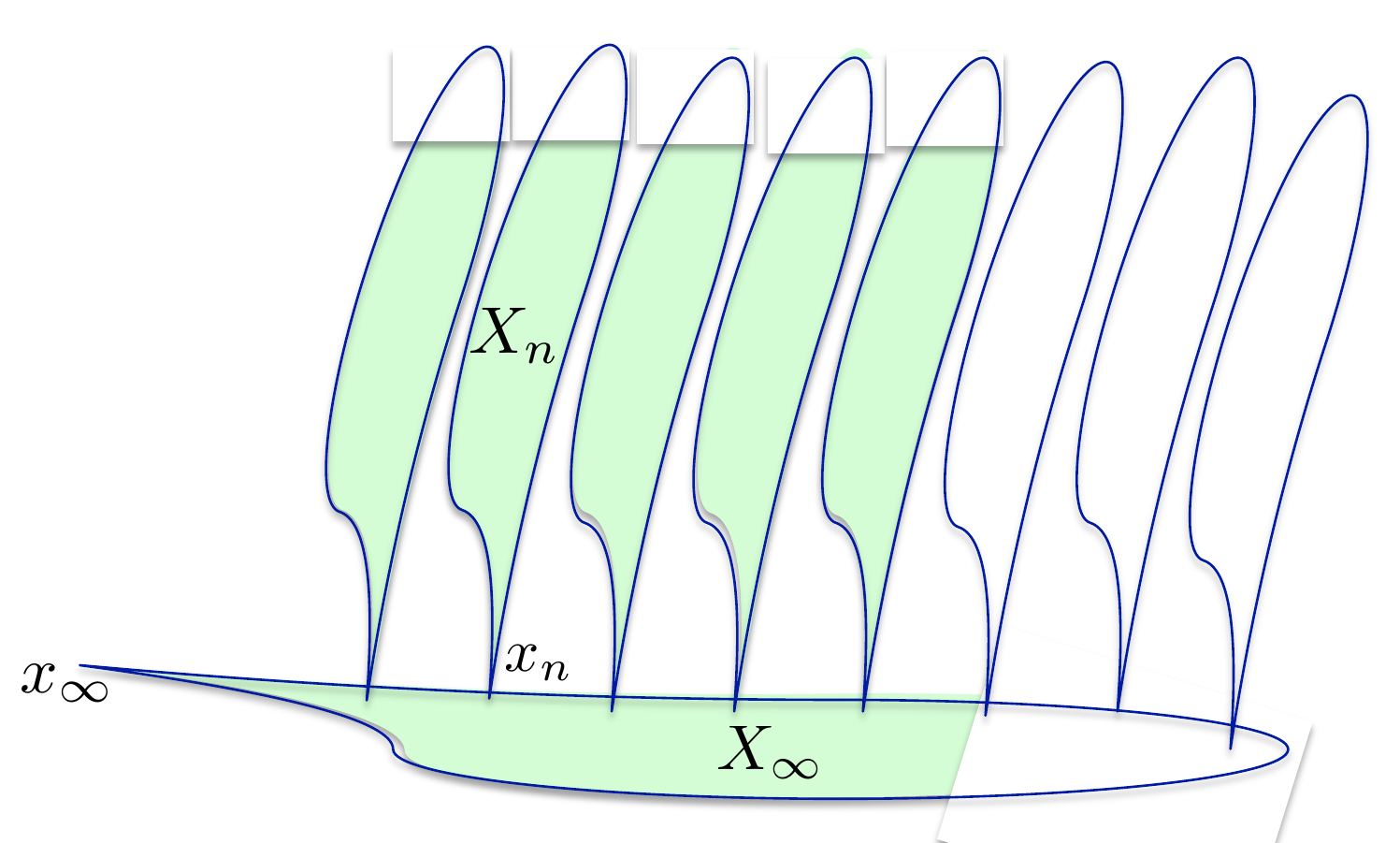}
\par\end{centering}
\caption{A typical open set $O$ containing $x_{\infty}$ is represented by
the green area; it has cofinite intersections with $X_{n}$ for a
cofinite subset of $x_{n}$ in $X_{\infty}$.}
\end{figure}
\end{example}

This pretopology is not a topology. Indeed, if $O$ is an open set
containing $x_{\infty}$, then $O\setminus X_{\infty}$ is finite,
and $O\setminus X_{n}$ is finite for each $n$ such that $x_{n}\in O\cap X_{\infty}$.
Therefore, $X_{\infty}\in\mathcal{V_{\theta}}(x_{\infty})\setminus\mathcal{O_{\theta}}(x_{\infty})$,
hence $\theta$ is not a topology.
\end{example}

\begin{example}[{A non-pretopological pseudotopology \cite[Example VIII.2.6]{CFT}}]
\label{exa:sequential}The \emph{usual sequential convergence} $\sigma$
of the real line is defined by $r\in\lm_{\sigma}\mathcal{F}$ if there
exists a sequential filter (Appendix \ref{subsec:Sequential-filters})
$\mathcal{E}\leq\mathcal{F}$ so that $\mathcal{E}$ converges to
$r$ in the \emph{usual topology}. The convergence $\sigma$ is not
a pretopology, because the infimum of all sequential filters converging
to a given point $r$ of $\mathbb{R}$, is the usual neighborhood
filter at $r$. In fact, if there were 
\begin{equation}
V\in\mathcal{V_{\sigma}}(r)\setminus\bigcap\{\mathcal{E}\in\mathcal{\mathbb{E}:E}\supset\mathcal{V_{\sigma}}(r)\},\label{eq:seq}
\end{equation}
where $\mathbb{E}$ stands for the class of sequential filters, then
there would be a sequential filter $\mathcal{E}\supset\mathcal{V_{\sigma}}(r)$
such that $V\notin\mathcal{E}$. This means that there exists a sequence
$(x_{n})_{n}$ of elements of $\mathbb{R}\setminus V,$ and which
generates a filter finer than $\mathcal{E}$, hence finer than $\mathcal{V_{\sigma}}(r)$
in contradiction with (\ref{eq:seq}).

The convergence $\sigma$ is a pseudotopology. Indeed, it is enough
to consider a free filter $\mathcal{F}$ such that $r\in\lm_{\sigma}\mathcal{U}$
for each $\mathcal{U}\in\beta\mathcal{F}$. Accordingly, for each
$\mathcal{U}\in\beta\mathcal{F}$ there is a countably infinite set
$E_{\mathcal{\mathcal{U}}}\in\mathcal{U}$, hence $\mathcal{E_{\mathcal{U}}\subset\mathcal{U}}$
and $r\in\lm_{\sigma}\mathcal{E_{\mathcal{U}}}$, where $\mathcal{E}_{\mathcal{\mathcal{U}}}$
is the cofinite filter of $E_{\mathcal{\mathcal{U}}}$ (see Appendix
\ref{subsec:Filters}). By Theorem \ref{thm:Comp_ultra}, there is
a finite subset $\boldsymbol{F}$ of $\beta\mathcal{F}$ such that
$\bigcup_{\mathcal{U}\in\boldsymbol{F}}E_{\mathcal{U}}\in\mathcal{F}$,
and thus $\mathcal{E}:=\bigwedge_{\mathcal{U}\in\boldsymbol{F}}\mathcal{E_{\mathcal{U}}}$
is the cofinite filter of $\bigcup_{\mathcal{U}\in\boldsymbol{F}}E_{\mathcal{U}}$,
hence $\mathcal{E}\subset\mathcal{F}$ (\footnote{Indeed, every free filter $\mathcal{F}$ is finer than the cofinite
filter of any $F_{0}\in\mathcal{F}$.}) and $r\in\lm_{\sigma}\mathcal{E}\subset\lm_{\sigma}\mathcal{F}$.
\end{example}

\begin{example}[Non-pseudotopological convergence]
\label{exa:non-pseudo} Consider a convergence on an infinite set
$X$, for which all the points are isolated, except for one $x_{\infty}$,
to which converges each free ultrafilter on $X$, as well as $\{x_{\infty}\}^{\uparrow}$.
This convergence is not pseudotopological, because the cofinite filter
$(X)_{0}$ (see Appendix \ref{subsec:Filters}), the infimum of all
free filters on $X$, does not converge to $x_{\infty}$, by the definition
of the convergence.
\end{example}

\subsection{Compactness\label{subsec:Compactness}}

In order to provide an opportune framework for further understanding,
I shall extend the notion of compactness in several ways. Let me recall
that no separation axiom is implicitly assumed!

In topology, a set $A$ is said to be \emph{compact at a set $B$}
if each open cover of $B$ admits a finite subfamily that is a cover
of $A$. This property admits an equivalent formulation in terms of
filters: if $\mathcal{F}$ is a filter such that $A\in\mathcal{F}^{\#}$
then $B\#\ad\mathcal{F}$.

A subset $K$ of a topological space is called \emph{compact} if it
is compact at $K$ (compact at itself). In each Hausdorff topology,
a compact subset is closed, but without the separation axiom (here
the Hausdorff property), this no longer the case (\footnote{\label{fn:sierp}Let $\$$ be a Sierpi\'{n}ski topology on $\{0,1\}$,
that is, the open sets are $\varnothing,\{0\},\{0,1\}$. In particular,
$\lm_{\$}\{0\}^{\uparrow}=\{0,1\}$. Each subset of this space is
compact, in particular, the set $\{0\}$ is compact, but is not closed.}).

In general convergences spaces, the two notion of compactness above
are no longer equivalent. Cover compactness is stronger than filter
compactness, and it turns out that filter compactness is more relevant,
because continuous maps preserve filter compactness, but, in general,
not cover compactness (see \cite{myn.relations}).

The concept of compactness has been extended from sets to families
of sets by \noun{L. V. Kantorovich} and al. \cite{KAN}, \noun{M.
P. Kac} \cite{kac}, and in \cite{DGL,DGL.kur} by \noun{G. H. Greco},
\noun{A. Lechicki} and the author.

A family $\mathcal{A}$ is said to be \emph{compact at} a family\emph{
$\mathcal{B}$} whenever if a filter meshes $\mathcal{A}$ then its
adherence meshes $\mathcal{B}$. A family is called \emph{compact}
if it is compact at itself; \emph{compactoid} (\footnote{The term \emph{compactoid} was introduced by \noun{G. Choquet} in
\cite{cho} for \emph{relatively compact}. Its advantage is its shortness.})\emph{ }if it is compact at the whole underlying space.

Obviously, the defined concept (\footnote{The concept of compactness of families of sets proved its usefulness
in various situations, for example, open sets in some hyperspaces
have been characterized in terms of compact families \cite[Theorem 3.1]{DGL.kur}. }) generalizes that of (relative) compactness: a set $A$ is compact
at a set $B$ if and only if $\{A\}$ is compact at $\{B\}$. 
\begin{rem*}
If a convergence $\theta$ is Hausdorff and $A$ is $\theta$-compact
at $B$, then $\ad_{\theta}A\subset B$. Indeed, if $x\in\ad_{\theta}A$,
then there exists an ultrafilter $\mathcal{U}$ such that $A\in\mathcal{U=\mathcal{U}^{\#}}$
and $\{x\}=\lm_{\theta}\mathcal{U}$, beacuse $\theta$ is Hausdorff.
But assumption, $B\#\ad_{\theta}\mathcal{U}=\lm_{\theta}\mathcal{U}$,
and thus $x\in B$. If a convergence is not Hausdorff, then $\ad_{\theta}A\subset B$
need not hold, as we have seen in the case of the Sierpi\'{n}ski topology
in Footnote \ref{fn:sierp}.
\end{rem*}
On the other hand, it extends the notion of convergence. In fact,
each convergent filter is compactoid, because $\lm_{\xi}\mathcal{F}\subset\ad_{\xi}\mathcal{H}$
for every $\mathcal{H}\#\mathcal{F}$. More significantly, $x\in\lm_{\mathrm{S}\xi}\mathcal{F}$
if and only if $\mathcal{F}$ is $\xi$-compact at $\{x\}$.

An ulterior extension allows to embrace \emph{countable compactness,
Lindelöf property, }and others.
\begin{defn}
\label{def:compact_at}Let $\mathbb{J}$ be a class of filter containing
all principal filters. Let $\xi$ be a convergence, and $\mathcal{A}$
and $\mathcal{B}$ be two families of subsets of $\left|\xi\right|$.
We say that $\mathcal{A}$ is $\mathbb{J}$\emph{-compact at $\mathcal{B}$
(in $\xi$) if
\[
\underset{\mathcal{F}\in\mathbb{J}}{\forall}(\mathcal{F}\#\mathcal{A}\Longrightarrow\ad_{\xi}\mathcal{F}\in\mathcal{B}^{\#}).
\]
}
\end{defn}

In particular, if $\mathbb{J}$ stands for the class of all filters,
we recover \emph{compactness}, and if $\mathbb{J}$ denotes the class
of countably based filters, we get \emph{countable compactness}. 

$\mathbb{J}$-compactness for the class $\mathbb{J}$ of principal
filters is a rather recent concept \cite{active}. It is called \emph{finite
compactness.} 
\begin{rem*}
In particular, if $A$ is finitely compact at $B$ for a $T_{1}$-convergence
(all points are closed), then $A\subset B$. Indeed, in this situation,
for every $x\in A$, the set $\{x\}$ meshes $A$, hence $\ad\{x\}=\{x\}\in B^{\#}$,
hence $x\in B$. Footnote \ref{fn:sierp} shows that this need not
be the case without the assumption of $T_{1}$. 
\end{rem*}
Let $\chi_{\xi}$ be the \emph{characteristic convergence} of $\xi$
(\footnote{The relation $\chi_{\xi}$ is a convergence. Indeed, if $x\in\lm_{\chi_{\xi}}\mathcal{F}$
and $\mathcal{F}\subset\mathcal{G}$, then $\varnothing\neq\lm_{\xi}\mathcal{F}\subset\lm_{\xi}\mathcal{G},$
hence $x\in X=\lm_{\chi_{\xi}}\mathcal{G}$. On the other hand, $x\in\lm_{\xi}\{x\}^{\uparrow}\subset\lm_{\chi_{\xi}}\{x\}^{\uparrow}$.}), defined by
\begin{equation}
\lm_{\chi_{\xi}}\mathcal{F}:=\begin{cases}
\left|\xi\right|, & \textrm{if }\lm_{\xi}\mathcal{F}\neq\varnothing,\\
\varnothing, & \textrm{otherwise.}
\end{cases}\label{eq:char_def}
\end{equation}

\begin{prop}
\label{prop:Comp_char}\cite{D.comp} If $J$ is $\mathbb{J}$-adherence-determined,
then a filter $\mathcal{H}$ is $\mathbb{J}$-compactoid (in $\xi$),
if and only if
\[
\lm_{J(\chi_{\xi})}\mathcal{H}\neq\varnothing.
\]
\end{prop}

\begin{proof}
By definition, $\mathcal{H}$ is $\mathbb{J}$-compactoid (in $\xi$)
if and only if $\ad_{\xi}\mathcal{J}\neq\varnothing$ (equivalently,
$\ad_{\chi_{\xi}}\mathcal{J}=X$) for each $\mathcal{J}\in\mathbb{J}$
such that $\mathcal{J}\#\mathcal{H}$. As
\[
\lm_{J(\chi_{\xi})}\mathcal{H}=\bigcap\nolimits _{\mathbb{J}\ni\mathcal{J}\#\mathcal{H}}\ad_{\chi_{\xi}}\mathcal{J},
\]
it is non-empty (equivalently, equal to the whole space) if and only
if $\ad_{\chi_{\xi}}\mathcal{J}=X$ for each $\mathcal{J}\in\mathbb{J}$.
\end{proof}
\begin{defn}
Let $\theta$ be a convergence on $W$ and $\sigma$ be a convergence
on $Z$. A relation $R$ is called $\mathbb{J}$\emph{-compact }if
$w\in\lm_{\theta}\mathcal{F}$ implies that $R(w)\#\ad_{\sigma}\mathcal{H}$
for each $\mathcal{H}\#R[\mathcal{F}]$ such that $\mathcal{H}\in\mathbb{J}$
(\footnote{A relation $R\subset W\times Z$ is $\mathbb{J}$\emph{-compactoid
}if $\lm_{\theta}\mathcal{F}\neq\varnothing$ implies that $R[\mathcal{F}]$
is $\mathbb{J}$-compactoid for $\sigma$. It is straightforward that
a graph-closed (Definition \ref{def:graph-closed}) $\mathbb{J}$-compactoid
relation is $\mathbb{J}$-compact. }).

Let us observe that (\footnote{Indeed, let $\mathcal{G}\in\mathbb{J}$ be such that $\mathcal{G}\#R[\mathcal{A}]$,
equivalently $R^{-}[\mathcal{G}]\#\mathcal{A}$, hence $\ad R^{-}[\mathcal{G}]\#B,$
because $R^{-}[\mathcal{G}]\in\mathbb{J}$. Therefore, there exists
$x\in\ad R^{-}[\mathcal{G}]\cap B$, so that there is a filter $\mathcal{F}\#R^{-}[\mathcal{G}]$,
equivalently, $R[\mathcal{F}]\#\mathcal{G},$ such that $B\#\lm\mathcal{F}$.
As $R$ is $\mathbb{J}$\nobreakdash-compact, $\ad\mathcal{G}\#Rx$,
hence $\ad\mathcal{G}\#RB$.})
\end{defn}

\begin{prop}
\label{prop:image_comp} Assume that the class $\mathbb{J}$ of filters
is transferable. Then the image under a $\mathbb{J}$\nobreakdash-compact
relation $R$ of a filter $\mathcal{A}$ that is $\mathbb{J}$-compact
at $B$ is $\mathbb{J}$-compact at $RB$.
\end{prop}

As it was observed after Proposition \ref{prop:transferable}, the
classes of all filters, countably based filters, and principal filters
are transferable.

\subsection{Completeness\label{subsec:Completeness}}

\emph{Compactness, local compactness}, and \emph{topological completeness}
(also called \emph{\v{C}ech completeness}) are special cases of $\kappa$\emph{-completeness}
(where $\kappa$ is a cardinal number), introduced and studied by
\noun{Z. Frolík} in \cite{Frolik_G}. In convergence spaces this notion
was investigated without any separation axioms, for example, in \cite{D.covers,compl_number,CFT}
(\footnote{There is a slight difference between our concept and that of Frolík
who uses filters admitting bases of open sets; the two concepts coincide
for regular topologies.}).

A convergence $\theta$ is called $\kappa$\emph{-complete} if there
exists a collection $\mathbb{P}$ of covers with $\mathrm{card}(\mathbb{P})\leq\kappa$
and such that $\ad_{\theta}\mathcal{G}\neq\varnothing$ for every
filter $\mathcal{G}$ such that $\mathcal{G}\cap\mathcal{P}\neq\varnothing$
for each $\mathcal{P}\in\mathbb{P}$. The property is not altered
if we consider merely collections $\mathbb{P}$ of \emph{ideals} (families
stable under subsets and finite unions).

By virtue of Theorem \ref{thm:duality}, a convergence $\theta$ is
$\kappa$-complete if and only if there exists a collection $\mathbb{H}$
of filters with $\mathrm{card}(\mathbb{H})\leq\kappa$ such that $\ad_{\theta}\mathcal{H}=\varnothing$
for each $\mathcal{H}\in\mathbb{H}$, and $\ad_{\theta}\mathcal{G}\neq\varnothing$
for every filter $\mathcal{G}$ such that $\neg(\mathcal{G}\#\mathcal{H})$
for every $\mathcal{H}\in\mathbb{H}$.

\begin{figure}[h]
\begin{centering}
\includegraphics[scale=0.35]{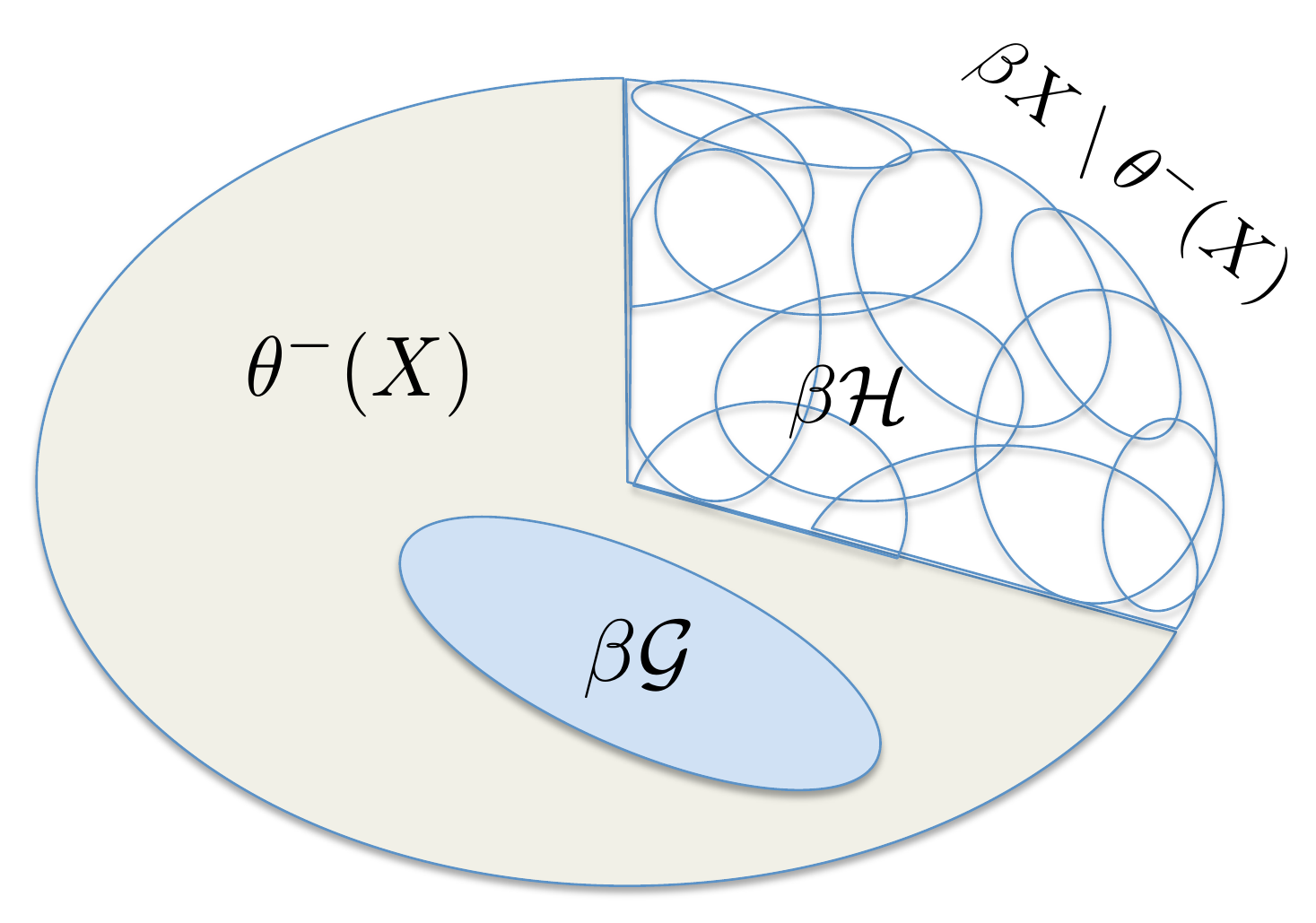}
\par\end{centering}
\caption{The set of non-$\theta$-adherent ultrafilters $\text{\ensuremath{\beta}}X\setminus\theta^{-}(X)$
is filled by $\beta\mathcal{H}:=\{\mathcal{U}\in\beta X:\mathcal{H}\protect\leq\mathcal{U}\}$,
where $\mathcal{H}\in\mathbb{H}$. Each filter $\mathcal{G}$ such
$\neg(\mathcal{G}\#\mathcal{H})$ for each $\mathcal{H}\in\mathbb{H}$
is $\theta$-adherent.}

\end{figure}

In other words, a convergence is $\kappa$-complete whenever there
exists a collection of cardinality $\kappa$ of non-adherent filters
such that an ultrafilter is convergent provided that it is not finer
than any filter from the collection.

The least $\kappa$ for which $\theta$ is $\kappa$-complete, is
called the \emph{completeness number} of $\theta$ and is denoted
by $\mathrm{compl}(\theta)$.
\begin{prop}
A convergence $\theta$ is compact if and only if $\mathrm{compl}(\theta)=0$.
\end{prop}

Indeed, the condition means that there is no non-$\theta$\nobreakdash-adherent
filter, that is, each ultrafilter is $\theta$\nobreakdash-convergent,
equivalently, each filter is $\theta$\nobreakdash-adherent.

A convergence is called \emph{locally compactoid} if each convergent
filter contains a compactoid set.
\begin{prop}
A convergence $\theta$ is locally compactoid if and only if $\mathrm{compl}(\theta)\leq1$
\emph{(}\footnote{\emph{If $\theta$ is $\kappa$-complete for $\kappa<\aleph_{0}$
then }$\mathrm{compl}(\theta)\leq1$.}\emph{)}.
\end{prop}

In fact, the condition says that there is a non-$\theta$\nobreakdash-adherent
filter $\mathcal{H}$ such that a filter $\mathcal{W}$ is $\theta$\nobreakdash-adherent
provided that there is $H\in\mathcal{H}$ with $H^{c}\in\mathcal{W}$.
Hence, the ideal $\mathcal{H}_{c}:=\{H^{c}:H\in\mathcal{H}\}$ is
composed of $\theta$\nobreakdash-compactoid sets. 

Therefore, if a filter $\mathcal{F}$ is $\theta$\nobreakdash-convergent,
then each $\mathcal{U}\in\beta\mathcal{F}$ contains an element of
$\mathcal{H}_{c}$, so that, by Theorem \ref{thm:Comp_ultra}, $\mathcal{F}\cap\mathcal{H}_{c}\neq\varnothing$,
because $\mathcal{H}_{c}$ is an ideal. As a result, $\mathcal{F}$
contains a $\theta$\nobreakdash-compactoid set, that is, $\theta$
is locally compactoid.

Finally, it follows from the definition of \v{C}ech completeness (e.
g., \cite{Frolik_G}) that
\begin{prop}
A convergence $\theta$ is \v{C}ech complete if and only if $\mathrm{compl}(\theta)\leq\aleph_{0}$.
\end{prop}

It is a classical result that a topology is completely metrizable
if and only if it is metrizable and \v{C}ech complete (e. g., \cite[p. 343]{Eng},\cite[Corollary XI.11.6]{CFT}).

\subsection{Examples of corefelective subclasses\label{subsec:Corefelective-subclasses}}

To complete our rough picture of convergence spaces, let us present
three basic coreflective properties. A convergence $\theta$ is called
\begin{enumerate}
\item \emph{sequentially based} if $x\in\lm_{\theta}\mathcal{F}$ entails
the existence of a sequential filter $\mathcal{E}$ (Appendix \ref{subsec:Sequential-filters})
such that $x\in\lm_{\theta}\mathcal{E}$ and $\mathcal{E\leq}\mathcal{F}.$
\item \emph{of countable character} if $x\in\lm_{\theta}\mathcal{F}$ entails
the existence of a countably based filter $\mathcal{E}$ such that
$x\in\lm_{\theta}\mathcal{E}$ and $\mathcal{E\leq}\mathcal{F}.$
M\emph{etrizable topologies} are of countable character.
\item \emph{locally compactoid} if $x\in\lm_{\theta}\mathcal{F}$ entails
the existence of a $\theta$-compactoid set $K\in\mathcal{F}.$
\end{enumerate}
The corresponding coreflectors are denoted by $\mathrm{Seq},\mathrm{I_{1}},$
and $\mathrm{K}$. The listed coreflective classes traverse the already
discussed reflective classes; they all contain the discrete topologies.
For instance, the convergence $\theta$ from Example \ref{exa:sequential}
is sequential.

\bibliographystyle{plain}
\bibliography{biblio2015,peanoSW_new29}

\newpage{}

\tableofcontents{}
\end{document}